\documentclass[11pt]{article}
\usepackage{geometry}
\makeatletter                   
\skip\footins=4ex               
\makeatother                    

\usepackage{amsmath}            
\usepackage{amsfonts}           
\usepackage{amssymb}
\usepackage{amsthm}
\usepackage{mathrsfs}
\usepackage{tikz-cd}
\usepackage[shortlabels]{enumitem}
\usepackage{tocbasic}
\usepackage{hyperref}
\usepackage{url}
\usepackage{listings}
\usepackage[T1]{fontenc}

\makeatletter

\newcommand{\BRANCHEDFACTOR}{0.85}

\newcommand*{\relrelbarsep}{.386ex}
\newcommand*{\relrelbar}{%
  \mathrel{%
    \mathpalette\@relrelbar\relrelbarsep
  }%
}
\newcommand*{\@relrelbar}[2]{%
  \raise#2\hbox to 0pt{$\m@th#1\relbar$\hss}%
  \lower#2\hbox{$\m@th#1\relbar$}%
}
\providecommand*{\rightrightarrowsfill@}{%
  \arrowfill@\relrelbar\relrelbar\rightrightarrows
}
\providecommand*{\leftleftarrowsfill@}{%
  \arrowfill@\leftleftarrows\relrelbar\relrelbar
}
\providecommand*{\xrightrightarrows}[2][]{%
  \ext@arrow 0359\rightrightarrowsfill@{#1}{#2}%
}
\providecommand*{\xleftleftarrows}[2][]{%
  \ext@arrow 3095\leftleftarrowsfill@{#1}{#2}%
}
\makeatother

\newcommand{\mb}{\overline{\mathcal M}}

\newcounter{problemcounter}
\newcounter{subproblemcounter}

\newcommand{\Z}{\ensuremath{\mathbb{Z}}}

\newcommand{\Aut}{\text{Aut}}

\newcommand{\Hb}{\overline{\mathcal H}}
\renewcommand{\a}{\mathfrak a}
\renewcommand{\b}{\mathfrak b}
\renewcommand{\P}{\mathbb P}
\theoremstyle{definition}
\newtheorem{thm}{Theorem}[section]

\newtheorem{dfn}[thm]{Definition}
\newtheorem{prop}[thm]{Proposition}
\newtheorem{lem}[thm]{Lemma}
\newtheorem{eg}[thm]{Example}

\usetikzlibrary{matrix,positioning,quotes}

\definecolor{myblue}{RGB}{120,94,240}
\definecolor{mygreen}{RGB}{255,194,10}
\definecolor{myred}{RGB}{0,158,115}
\definecolor{mypink}{RGB}{220,38,127}

\title{Reverse Hurwitz counts of genus $1$ curves}

\author{Michael Mueller}

\begin{document}

\maketitle

\abstract{
    In this paper, we study a problem that is in a sense a reversal of the Hurwitz
counting problem. The Hurwitz problem asks: for a
generic {\it target} --- $\P^1$ with a list of $n$ points  $q_1,\dots,q_n\in \P^1$ --- and
partitions $\sigma_1,\dots,\sigma_n$ of $d$, how many degree
$d$ covers $C\to\P^1$ are there with specified ramification $\sigma_i$ over $q_i$?
We ask: for a generic {\it source} --- an $r$-pointed curve
$(C,p_1,\dots,p_r)$ of genus $1$ --- and
partitions $\mu,\sigma_1,\dots,\sigma_n$ of $d$ with $\ell(\mu)=r$,
how many degree $d$ covers $C\to\P^1$ are there with ramification profile $\mu$ over $0$
corresponding to a fiber $\{p_1,\dots,p_r\}$ and
elsewhere ramification profiles $\sigma_1,\dots,\sigma_n$?

While the enumerative invariants we study bear a similarity to
generalized Tevelev degrees, they are more difficult to express in
closed form in general. Nonetheless, we establish key results:
after proving a closed form result in the case where
the only non-simple unmarked ramification profiles $\sigma_1$ and $\sigma_2$ are ``even'' (consisting of $2,\dots,2$), we go on to establish
recursive formulas to compute invariants where each unmarked ramification
profile is of the form $(x,1,\dots,1)$.
A special case asks: given a generic $d$-pointed genus $1$ curve $(E,p_1,\dots,p_d)$, how many degree $d$ covers $(E,p_1,\dots,p_d)\to(\P^1,0)$ are there with $d-2$ unspecified points of $E$ having ramification index $3$? We show that the answer is an explicit quartic in $d$.

}

\tableofcontents

\section{Introduction}

\subsection{Admissible covers}

We pose the question: given an $n$-pointed curve
$(C,p_1,\dots,p_r)$ of genus $g$ and
partitions $\mu,\sigma_1,\dots,\sigma_n$ of $d$ with $\ell(\mu)=r$,
how many degree $d$ covers $C\to\P^1$ are there with ramification profile $\mu$ over $0$
corresponding to a fiber $\{p_1,\dots,p_r\}$ and
elsewhere ramification profiles $\sigma_1,\dots,\sigma_n$?
In this paper we mostly specialize to the case $g=1$.

\begin{eg}
  \label{eg:firsteg}
Given a generic elliptic curve $(E,p)$, how many degree $4$ covers $E\to\P^1$ are there
with $p$ fully ramified over $0$ and ramification profiles $(2,2)$, $(2,2)$ and $(2,1,1)$
elsewhere?
(See Figure~\ref{fig:firsteg}; only the pink, bolded point $p$ is fixed.)
\end{eg}

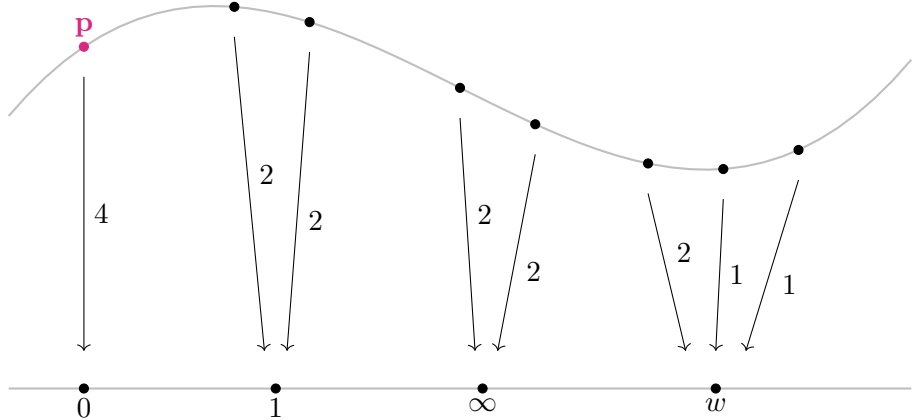
\begin{figure}[h]
  \caption{Degree $4$ covers $(E,p)\to(\P^1,0)$ with full ramification at $p$ and ramification profiles $(2,2)$, $(2,2)$ and $(2,1,1)$ elsewhere.}
  \centering
\begin{tikzpicture}
  \draw[domain=-6:6,samples=50,color=gray!50,thick] plot (\x, \x^3/64 - \x/2);
  \foreach \x/\p/\l in {-5/$\mathbf{p}$/4}
  {
    \filldraw[mypink] (\x,\x^3/64 - \x/2) circle (0.6mm) node[above] {\p};
    \draw[->] (\x,\x^3/64 - \x/2 - .4) to["$\l$"] (-5, -3.5);
  }
  \foreach \x/\p/\l in {-3//2, -2//2}
  {
    \filldraw[black] (\x,\x^3/64 - \x/2) circle (0.6mm) node[above] {\p};
    \draw[->] (\x,\x^3/64 - \x/2 - .4) to["$\l$"] (-1.7+.3*\x, -3.5);
  }
  \foreach \x/\p/\l in {0//2, 1//2}
  {
    \filldraw[black] (\x,\x^3/64 - \x/2) circle (0.6mm) node[above] {\p};
    \draw[->] (\x,\x^3/64 - \x/2 - .4) to["$\l$"] (.2+.3*\x, -3.5);
    }
  \foreach \x/\p/\l in {2.5//2, 3.5//1, 4.5//1}
  {
    \filldraw[black] (\x,\x^3/64 - \x/2) circle (0.6mm) node[above] {\p};
    \draw[->] (\x,\x^3/64 - \x/2 - .4) to["$\l$"] (2+.4*\x, -3.5);
    }

    \draw[domain=-6:6,samples=50,color=gray!50,thick] plot (\x, -4);
  \foreach \x/\q in {-5/$0$, -2.45/$1$, .3/$\infty$, 3.4/$w$}
    \filldraw[black] (\x,-4) circle (0.6mm) node[below] {\q};
\end{tikzpicture}

\label{fig:firsteg}
\end{figure}

As we will need to count covers,
we begin by reviewing some relevant background in the theory of
(admissible) covers and enumerative counts known in the literature.

Recall the moduli of stable $n$-pointed genus $g$ curves, known as $\overline{\mathcal M}_{g,n}$ (see \cite{Moduli} for background).
The properties of stable pointed curves are often reflected in their associated dual graphs (defined in \cite{Vakil}, p.\ 6).

In order to count covers $C\to\P^1$, it will be useful
to have a compact moduli space of covers of $\P^1$ to work with.
There are multiple ways to compactify moduli of smooth covers;
one common one is the moduli of stable maps $\overline{\mathcal M}_{g,n}(\P^1,\beta)$
(see \cite{StableMaps} for notes on the subject). While this moduli
space is frequently used in Gromov-Witten theory,
we will not use it here.
As in \cite{Generalized} we will use the {\it moduli of admissible covers}, originally
treated in \cite{Admissible}. Admissible covers are defined in \cite{Moduli} (Definition 3.149).
Letting
$\sigma=(\sigma_1,\dots,\sigma_n)$ be a list of
partitions of a fixed integer $d$, we consider the
Hurwitz moduli space $\Hb_{g,\sigma}$ consisting of admissible covers
$C\to X$ and marked points $p_1,\dots,p_N\in C$ where $C$ is
of genus $g$, $X$ is of genus $0$,
$N=\sum_{i=1}^n\ell(\sigma_i)$, and the marked points
comprise the ramified fibers with profiles given by $\sigma_1,\dots,\sigma_n$.

For an admissible cover of curves $C\to X$, there are associated dual
graphs $\Gamma$ to $C$ and $\Gamma'$ to $X$. Between these two graphs
there is a type of map, known as an admissible cover of graphs:

\begin{dfn}[\cite{Lian}, Definition 2.10]
  An admissible cover of stable graphs $\gamma:\Gamma\to\Gamma'$
  of degree $d$ is a collection of maps on vertices, half-edges, and
  legs compatible with all of the attachment
  data, in addition to the data of a degree
  $d_v$ at each vertex $v\in V(\Gamma)$,
  and a ramification index $d_e$ at each
  edge $e\in E(\Gamma)$ such that:
  \begin{itemize}
  \item if $v\in V(\Gamma)$ and $h'\in H(\Gamma')$
    is a half-edge attached to $\gamma_V(v)$, then
    the sum of the ramification indices
    at the half-edges attached to $v$ living
    over $h'$ is equal to $d_v$, and
  \item
    if $v'\in V(\Gamma')$, then the sum
    of the degrees at vertices over $v'$ is
    equal to $d$.
    \end{itemize}
\end{dfn}

\begin{eg}

  We depict an admissible map of graphs $\Gamma\to\Gamma'$, where the upper three vertices are
  in graph $\Gamma$ and the lower two vertices are in graph $\Gamma'$.

  \begin{figure}[h]
  \caption{Illustration of an admissible cover of graphs.}
  \centering
    \scalebox{\BRANCHEDFACTOR}{
      \begin{tikzpicture}[thick,amat/.style={matrix of nodes,nodes in empty cells,
  row sep=2.5em,rounded corners,
  nodes={draw,solid,circle,minimum size=1.0cm}},
  dmat/.style={matrix of nodes,nodes in empty cells,row sep=2.5em,nodes={minimum size=1.0cm},draw=myred},
  fsnode/.style={fill=myblue},
  ssnode/.style={fill=mygreen}]

                                      \matrix[amat,nodes=fsnode] (mat1) {$1$\\
                                      $0$\\};

  \matrix[dmat,left=1cm of mat1] (degrees1) {$2$\\
  $1$\\};

 \matrix[amat,right=4cm of mat1,nodes=ssnode] (mat2) {$0$\\};

 \matrix[dmat,right=1cm of mat2] (degrees2) {$3$\\};

 \draw  (mat1-1-1) edge["$2$"] (mat2-1-1)
 (mat1-2-1) edge["$1$"] (mat2-1-1);

 \draw (mat1-2-1) -- +(160:.7) node[anchor=east] {$p_1$}
 (mat1-2-1) -- +(45:.7) node[anchor=west] {$p_2$};

 \draw (mat2-1-1) -- +(345:.7) node[anchor=west] {$p_{3}$};

 \matrix[amat,nodes=fsnode,below=2cm of mat1] (mat3) {$0$\\};

 \matrix[amat,nodes=ssnode,below=3cm of mat2] (mat4) {$0$\\};

  \draw (mat3-1-1) -- +(160:.7) node[anchor=east] {$q_1$}
  (mat3-1-1) -- +(45:.7) node[anchor=west] {$q_2$};
   
  \draw (mat4-1-1) -- +(345:.7) node[anchor=west] {$q_3$};

 \draw  (mat3-1-1) edge[] (mat4-1-1);

 \draw  (mat1-2-1) edge[->,shorten <= 0.3cm, shorten >= 0.3cm] (mat3-1-1);
 \draw  (mat2-1-1) edge[->,shorten <= 0.3cm, shorten >= 0.3cm] (mat4-1-1);

      \end{tikzpicture}}
  \end{figure}
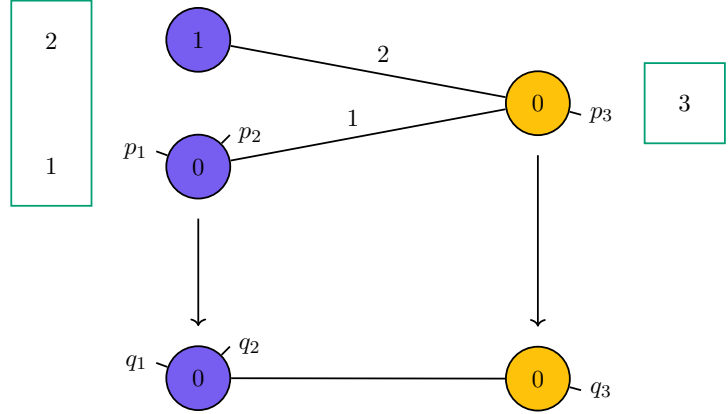

                                      The numbers in green boxes are the degrees $d_v$ for each vertex of the
                                      source, while the labels of edges
                                      are the ramification indices $d_e$.
                                      Note that the sum of degrees of
                                      vertices over the purple target
                                      vertex is $2+1=3$. (In future, we will omit the arrows, and use the convention that vertices in the source map to the vertex in
  the target of the same color. When no target is drawn, we assume it can be described by a purple and an orange vertex with an edge between them.)

\end{eg}

We will also refer later to the combinatorial concept
of an $A$-structure on $\Gamma$ where $A$ and $\Gamma$
are stable graphs with the same leg set, to describe
specializations of a curve. This notion is
defined in \cite{Schmitt} (Definition 2.5).

\subsection{Hurwitz theory}

For partitions $\sigma_1,\dots,\sigma_n$ of $d$
satisfying the Riemann-Hurwitz condition, there is a natural map
\[
\epsilon_0:\Hb_{g,\sigma}\to\overline{\mathcal M}_{0,n}
\]
Hurwitz theory (\cite{Cela}) tells us that this map is finite and we have a marked connected Hurwitz number
\[\overline H^{connected}(\sigma_1,\dots,\sigma_n)=\text{deg}(\epsilon_0)\]
The unmarked Hurwitz number is obtained by dividing by automorphisms of the ramification points upstairs:
\[
H^{connected}(\sigma_1,\dots,\sigma_n)=\frac{\overline H^{connected}(\sigma_1,\dots,\sigma_n)}{|\Aut(\sigma_1,\dots,\sigma_n)|}
\]
There is a more general notion of Hurwitz number $H(\sigma_1,\dots,\sigma_n)$, in which the source of
the cover may be disconnected.
The following well-known characterization can be used to produce the results
 in Appendix~\ref{appendix:hurwitz}:
\begin{prop}[\cite{Completed}, p.\ 523]
  Let $\sigma_1,\dots,\sigma_n$ be partitions of $d$. The Hurwitz number $H(\sigma_1,\dots,\sigma_n)$ is given by
  \[
  H(\sigma_1,\dots,\sigma_n)=\frac 1{d!}\cdot |\{(s_1,\dots,s_n)\in S_d^n:s_i\text{ has cycle type }\sigma_i,\ s_1\cdots s_n=1\in S_d\}|
  \]
  and the connected version $H^{connected}(\sigma_1,\dots,\sigma_n)$ is defined
  similarly, except that the subgroup of $S_d$ generated by $s_1,\dots,s_n$
  must act transitively on $\{1,\dots,d\}$.
\end{prop}

\subsection{Generalized Tevelev degree}

Whereas Hurwitz numbers count covers of $\P^1$ with fixed branch points,
{\it generalized Tevelev degrees} \cite{Generalized} fix
a pointed source curve and certain branch points of the target $\P^1$ before counting.
For instance:

\begin{eg}
  \label{eg:tevelev}
  Fix a generic genus $1$ curve $C$ with points $p_1,\dots,p_6\in C$. Also fix generic points
  $q_1,\dots,q_4\in\P^1$. We can consider degree $4$ maps $C\to\P^1$ where $p_1$ and $p_2$ are unramified points in the same fiber over $q_1$, $p_3$ and $p_4$ are unramified points in the same fiber over $q_2$, and $p_5$ and $p_6$ are ramified of index $2$ over the points $q_3$ and $q_4$ respectively.

  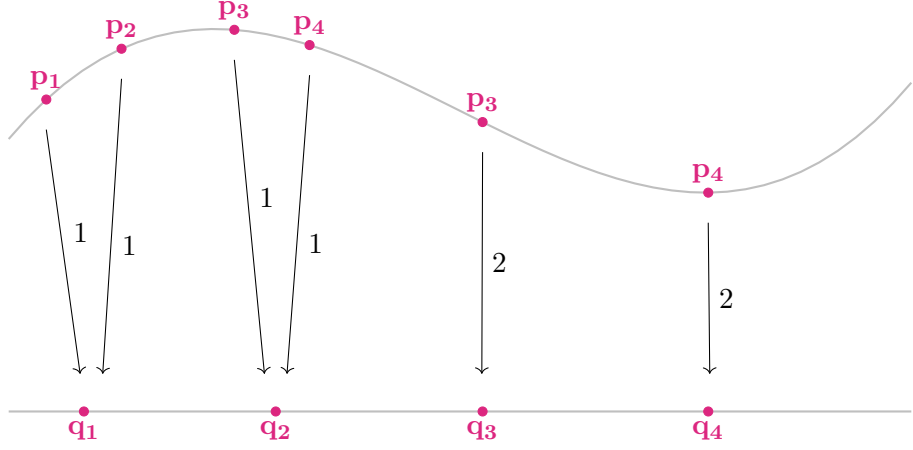
\begin{figure}[h]
  \caption{Illustration of the covers in Example~\ref{eg:tevelev}. Note that each depicted fiber has two unramified points not shown, and there are $5$ simply branched points not shown. Each pink, bolded point is fixed.}
  \centering
  \begin{tikzpicture}
  \draw[domain=-6:6,samples=50,color=gray!50,thick] plot (\x, \x^3/64 - \x/2);
  \foreach \x/\p/\l in {-5.5/$\mathbf{p_1}$/1, -4.5/$\mathbf{p_2}$/1}
  {
    \filldraw[mypink] (\x,\x^3/64 - \x/2) circle (0.6mm) node[above] {\p};
    \draw[->] (\x,\x^3/64 - \x/2 - .4) to["$\l$"] (-3.4+.3*\x, -3.5);
  }
  \foreach \x/\p/\l in {-3/$\mathbf{p_3}$/1, -2/$\mathbf{p_4}$/1}
  {
    \filldraw[mypink] (\x,\x^3/64 - \x/2) circle (0.6mm) node[above] {\p};
    \draw[->] (\x,\x^3/64 - \x/2 - .4) to["$\l$"] (-1.7+.3*\x, -3.5);
  }
  \foreach \x/\p/\l in {0.3/$\mathbf{p_3}$/2}
  {
    \filldraw[mypink] (\x,\x^3/64 - \x/2) circle (0.6mm) node[above] {\p};
    \draw[->] (\x,\x^3/64 - \x/2 - .4) to["$\l$"] (.2+.3*\x, -3.5);
    }
  \foreach \x/\p/\l in {3.3/$\mathbf{p_4}$/2}
  {
    \filldraw[mypink] (\x,\x^3/64 - \x/2) circle (0.6mm) node[above] {\p};
    \draw[->] (\x,\x^3/64 - \x/2 - .4) to["$\l$"] (2+.4*\x, -3.5);
    }

    \draw[domain=-6:6,samples=50,color=gray!50,thick] plot (\x, -4);
  \foreach \x/\q in {-5/$\mathbf{q_1}$, -2.45/$\mathbf{q_2}$, .3/$\mathbf{q_3}$, 3.3/$\mathbf{q_4}$}
    \filldraw[mypink] (\x,-4) circle (0.6mm) node[below] {\q};
\end{tikzpicture}

\label{fig:tevelev}
\end{figure}

  In the notation of \cite{Generalized} we have $g=1$, $d=4$, $\ell=2$, $k=4$, $\mu_1=\mu_2=(1,1)$, $\mu_3=\mu_4=(2)$ and $n=6$. By Theorem 9 of that paper, the enumerative count is
  \[
  \text{Tev}_{1,2,(1,1),(1,1),(2),(2)}=2^1=2.
  \]
\end{eg}

\subsection{Problem statement}

In order to study the problem posed in the abstract, we will find it useful to set up a more general problem, in which
multiple fibers in the source curve may contain fixed points:

Fix a positive integer $d$, a nonnegative integer $g$, ordered partitions $\sigma_1,\dots,\sigma_n$ of $d$, and subpartitions $\nu_1,\dots,\nu_n$ (i.e., $\nu_i$ is a subset of $\sigma_i$).
Let $\Hb_{g,\sigma}$ be the moduli space consisting of degree $d$ admissible covers $C\to X$ where $C$ is a genus $g$ prestable curve and $X$ is a genus $0$ prestable curve, marked points $p_{i,j}\in C$ for every $i\in\{1,\dots,n\}$ and
$j\leq \ell(\sigma_i)$, 
marked points $q_1,\dots,q_n\in X$ such that
the fiber over $q_i$ is $\{p_{i,1},\dots,p_{i,\ell(\sigma_i)}\}$ with
ramification profile $\sigma_i$, and the map is \'etale away from $q_1,\dots,q_n$ and any nodes of the target. Our situation is described in Figure \ref{fig:hurwitz}.

Let $m=\sum\ell(\nu_i)$ and $N=\sum\ell(\sigma_i)$.
We have natural maps:

~\\
\begin{figure}[h]
  \caption{Maps from the Hurwitz moduli space}
  \centering
  \[\begin{tikzcd}
	{\overline{\mathcal H}_{g,\sigma}} & {\overline{\mathcal M}_{g,N}} \\
	{\overline{\mathcal M}_{0,n}} & {\overline{\mathcal M}_{g,m}}
	\arrow["{\lambda_g}", from=1-1, to=1-2]
	\arrow["{\epsilon_0}"', from=1-1, to=2-1]
	\arrow["{\pi_{\nu}}"', from=1-1, to=2-2]
	\arrow["\iota", from=1-2, to=2-2]
\end{tikzcd}\]
\label{fig:maps}
\end{figure}

Here $\pi_{\nu}$ remembers the source curve and the points corresponding to $\nu$; $\epsilon_0$ remembers the target with its marked
branch points; $\lambda_g$ remembers the source curve and all points
corresponding to $\sigma$; and $\iota$ is the forgetful map, which forgets
the points of the source corresponding to $\sigma-\nu$ and stabilizes.

\begin{lem}
  \label{lem:dim}
  The source and target of $\pi_{\nu}$ are of the same dimension when $3g+m=n$,
  and thus if $\pi_{\nu}$ is finite this condition must hold.
\end{lem}
\begin{proof}
  By Hurwitz theory, we know that the map $\epsilon_0$ is finite, and so $\dim(\Hb_{\sigma,\mu})=n-3$. The condition that
  $\dim(\Hb_{\sigma,\mu})=\dim(\mb_{g,m})$ is equivalent to
  \[
  3g-3+m=n-3\iff 3g+m=n.
  \]
\end{proof}

\begin{figure}[h]
  \caption{Depiction of a Hurwitz cover with notation used in this paper. Pink, bolded points are the ones fixed when counting $N_{g,\sigma,\nu}$.}
  \centering
\begin{tikzpicture}
  \draw[domain=-6:6,samples=50,color=gray!50,thick] plot (\x, \x^3/64 - \x/2);
  \foreach \x/\p/\c in {-5.5/$\mathbf{p_{1,1}}$/mypink, -4.5/$\mathbf{p_{1,\ell(\nu_1)}}$/mypink, -3/$p_{1,\ell(\sigma_1)}$/black}
  {
    \filldraw[\c] (\x,\x^3/64 - \x/2) circle (0.6mm) node[above] {\p};
    \draw[->] (\x,\x^3/64 - \x/2 - .4) -- (-3.15+.3*\x, -3.5);
    }
  \foreach \x/\p/\c in {2.5/$\mathbf{p_{n,1}}$/mypink, 3.5/$\mathbf{p_{n,\ell(\nu_n)}}$/mypink, 4.5/$p_{n,\ell(\sigma_n)}$/black}
  {
    \filldraw[\c] (\x,\x^3/64 - \x/2) circle (0.6mm) node[above] {\p};
    \draw[->] (\x,\x^3/64 - \x/2 - .4) -- (2+.3*\x, -3.5);
    }

    \draw[domain=-6:6,samples=50,color=gray!50,thick] plot (\x, -4);
  \foreach \x/\q in {-4.5/$q_1$, 3/$q_n$}
    \filldraw[black] (\x,-4) circle (0.6mm) node[below] {\q};
\end{tikzpicture}

\label{fig:hurwitz}
\end{figure}
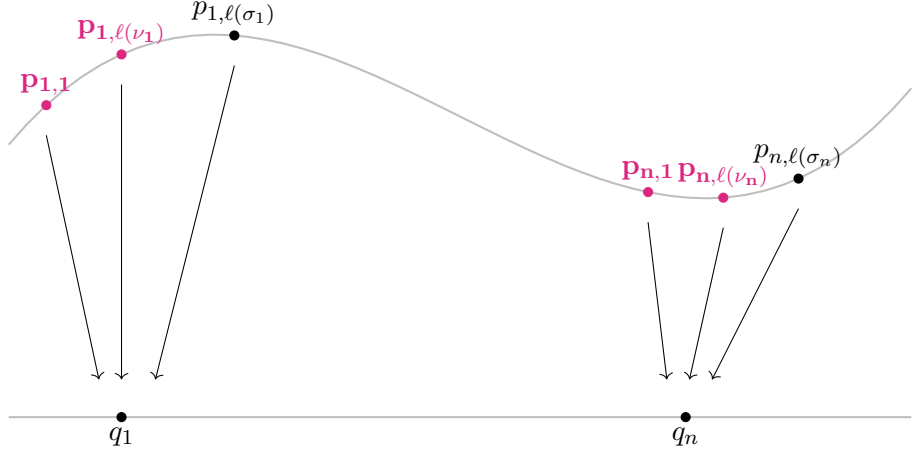

When the condition of Lemma~\ref{lem:dim} is satisfied, we can define:
\begin{dfn}
  For a genus $g$, list of partitions $\sigma$ of the
  same positive integer $d$, and subpartitions $\nu$,
\[
W_{g,\sigma,\nu}=\deg(\pi_{\nu})
\]
where this is an orbifold degree.
\end{dfn}
In order to account for relabeling of marked points (any of the points $q_n$ with identical data, as well as any of the unmarked points in fibers with
identical data), we define
\[
N_{g,\sigma,\nu}=\frac{W_{g,\sigma,\nu}}{|\Aut(\sigma-\nu)|}
\]
where $\Aut(\sigma-\nu)$ denotes automorphisms of $\sigma_i-\nu_i$ for each $i$ together with automorphisms of $\{\sigma_i-\nu_i\}$, i.e.
\[
|\Aut(\sigma-\nu)|=|\Aut(\{(\sigma_i,\nu_i):1\leq i\leq n\})|\cdot \prod_i|\Aut(\sigma_i-\nu_i)|.
\]
\begin{eg}
  When $g=1$, $\sigma=((2,2),(3,1),(2,2),(2,1,1),(2,1,1))$ and $\nu=((2,2),\emptyset,\emptyset,\emptyset,\emptyset)$, we have
  \[
  N_{g,\sigma,\nu}=\frac{W_{g,\sigma,\nu}}{|\Aut(3,1)|\cdot |\Aut(2,2)|\cdot |\Aut(2,1,1)|^2\cdot |\Aut((3,1),(2,2),(2,1,1),(2,1,1))|}=\frac{W_{g,\sigma,\nu}}{1\cdot 2\cdot 2^2\cdot 2}
  \]
  Here $\Aut(x_1,\dots,x_k)$ consists of permutations $\tau:\{1,\dots,k\}\to\{1,\dots,k\}$ such that
  $x_{\tau(i)}=x_i$.
  \end{eg}

One way to think about $N_{g,\sigma,\nu}$ is as follows: the map
$\pi_{\nu}:\Hb_{g,\sigma}\to\mb_{g,m}$ factors through
$\Hb_{g,\sigma}/\Aut(\sigma-\nu)$, and $N_{g,\sigma,\nu}$ is the degree of the map
\[
\Hb_{g,\sigma}/\Aut(\sigma-\nu)\to\mb_{g,m}.
\]

We will assume when not otherwise specified that $g=1$, and will also often deal with the case where
$\sigma_1=\nu_1=\mu$ and $\nu_3=\nu_4=\dots=\emptyset$.
\begin{dfn}
  Write
\[
W_{\sigma_1,\dots,\sigma_n}(\mu)=W_{1,(\mu,\sigma_1,\dots,\sigma_n),(\mu,\emptyset,\dots,\emptyset)}
\]
for the relevant labeled invariant, and
\[
N_{\sigma_1,\dots,\sigma_n}(\mu)=N_{1,(\mu,\sigma_1,\dots,\sigma_n),(\mu,\emptyset,\dots,\emptyset)}
\]
for the unlabeled invariant.\end{dfn}
In particular,
\[
N_{\sigma_1,\dots,\sigma_n}(\mu)=\frac{W_{\sigma_1,\dots,\sigma_n}(\mu)}{|\Aut(\sigma_2,\dots,\sigma_n)|\cdot\prod_{i=2}^n|\Aut(\sigma_i)|}
\]
If $\mu,\sigma_1,\dots,\sigma_n$ do not have ramification sufficient to satisfy Riemann-Hurwitz,
we let \[N_{\sigma_1,\dots,\sigma_n}(\mu) = N_{\sigma_1,\dots,\sigma_n,(2,1^{d-2}),\dots}(\mu).\]
The primary task of this paper is to understand the invariants $N_{\sigma_1,\dots,\sigma_n}(\mu)$
in some significant special cases.

\subsection{Comparison of invariants}

Hurwitz numbers, generalized Tevelev degrees, and our invariants $N_{g,\sigma,\nu}$
all enumerate covers of $\P^1$. We review the differences between them, using the following diagram as reference:

\[\begin{tikzcd}[ampersand replacement=\&]
	{\overline{\mathcal H}_{g,\sigma}} \& {\overline{\mathcal M}_{g,N}} \\
	{\overline{\mathcal M}_{0,n}} \& {\overline{\mathcal M}_{g,m}}
	\arrow["{\lambda_g}", from=1-1, to=1-2]
	\arrow["{\epsilon_0}"', from=1-1, to=2-1]
	\arrow["{\pi_{\nu}}"', from=1-1, to=2-2]
	\arrow["\iota", from=1-2, to=2-2]
\end{tikzcd}\]

\begin{itemize}
\item Hurwitz numbers ($H=\deg(\epsilon_0)$) count covers with a fixed {\bf target};
\item Generalized Tevelev degrees count covers with a partially fixed {\bf target} and a partially fixed {\it source};
\item Our invariants $N_{g,\sigma,\nu}=\frac{\deg(\pi_{\nu})}{|\Aut(\sigma-\nu)|}$ count Hurwitz covers with a partially fixed {\it source}.
  
\end{itemize}

Section 1.5 of \cite{Cela} notes that
Hurwitz numbers and generalized Tevelev degrees both give partial information about the pushforward $\tau_*[\overline{\mathcal H}_{g,\sigma}]$, where
                                                                                                                                                                                                                            $\tau=(\lambda_g,\epsilon_0):\overline{\mathcal H}_{g,\sigma}\to\overline{\mathcal M}_{g,N}\times\overline{\mathcal M}_{0,n}$. We can see that the invariants $N_{g,\sigma,\nu}$ do so as well.

\subsection{Main results}

In Section~\ref{section:twos}, we study a generalization of Example~\ref{eg:firsteg}: given
  an even degree $d=2k$, a partition $\mu=(\mu_1,\dots,\mu_r)$ of $d$, and a generic
  $r$-pointed genus $1$ curve $(C,p_1,\dots,p_r)$, we ask how many maps
  $f:C\to\P^1$ exist with $f^{-1}(0)=\{p_1,\dots,p_r\}$, ramification index $\mu_i$
  at $p_i$, and ramification profiles $(2,\dots,2)$ over both $1$ and $\infty\in\P^1$.

  \begin{thm}
    \label{thm:twos}
      Let $d=2k$ be any even integer greater than $2$. For any $\mu$ with sum $|\mu|=d$, $N_{(2^k),(2^k)}(\mu)=3\cdot 2^{\ell(\mu)-1}$.
  (In particular, there are $3$ maps of
    the kind considered in Example~\ref{eg:firsteg}.)
    \end{thm}
  
  In Section~\ref{section:Sab}, we study a different problem: given
  a degree $d$, a partition $\mu=(\mu_1,\dots,\mu_k)$ of $d$, multisets of positive integers $\a=\{a_1,\dots,a_r\}$ and $\b=\{b_1,\dots,b_s\}$,
  and a generic $(k+s)$-pointed genus $1$ curve $(C,p_1,\dots,p_k,P_1,\dots,P_s)$,
  we ask how many maps $f:C\to\P^1$ exist with $f^{-1}(0)=\{p_1,\dots,p_k\}$, ramification
  index $\mu_i$ at $p_i$, ramification index $b_j$ at $P_j$, and other (unmarked) points
  with ramification indices $a_1,\dots,a_r$. (Call the answer $S_{\a,\b}(\mu)$.)
  While we do not find a closed form solution, we develop recursive relations that suffice
  for effective computation:

  \begin{thm}
    \label{thm:Sabcomputable}
  Every invariant $S_{\a,\b}(\mu)$ can be computed by
  a finite
  sequence of applications of Theorems~\ref{thm:reduceb}, \ref{thm:reducea} and \ref{thm:Sabbase}.
  \end{thm}

  While our primary interest is in the case when $\b=\emptyset$, in keeping with the question posed in the abstract,
  Theorem~\ref{thm:Sabcomputable} makes use of the more general invariants to recursively compute $S_{\a,\emptyset}(\mu)$.
  We will establish a closed form in an interesting special case:

  \begin{thm}
    \label{thm:deg4}
    Given a degree $d\geq 2$ and a generic $d$-pointed genus $1$ curve $(C,x_1,\dots,x_d)$,
    the number of degree $d$ maps $f:C\to\P^1$ such that $f^{-1}(0)=\{p_1,\dots,p_d\}$ and
    $d-2$ points of $C$ have ramification index $3$ is given by
    \[\frac{d(d-1)(2(d-1)^2+1)}{6}\]
    \end{thm}

The primary technique we use to count covers of interest is to
count covers over a non-generic pointed {\it stable} curve. This method is not new; it was used originally in \cite{Cela} thanks to a technical result of \cite{Lian}, and is also used in \cite{Generalized}. We adapt this method to our context in Section~\ref{section:admissible}, allowing us to turn our central question into a combination of graph combinatorics and solving recurrence relations. It is also sometimes possible to compute our invariants through direct evaluation, by using the line bundle characterization of maps to $\P^1$ to recast the problem as an enumerative question about subvarieties of projective space, but we do not use this method here.

A few potential next steps could be pursued to build on our work.
One would be to study invariants $N_{g,\sigma,\nu}$ with
$g>1$ (or $g=0$); the technique in Section~\ref{section:admissible} will continue
to apply, though with some modifications. Another would be to establish results about the
``virtual'' version of the invariants studied here, in which admissible covers are replaced with stable maps in the definitions.

\subsection{Acknowledgements}

This paper is an adaptation of my dissertation completed as a PhD student at the University of Michigan \cite{Thesis}.
I received support from NSF RTG grant DMS 1840234 in summer and fall of 2022.

My advisor, Aaron Pixton, played an essential role as he guided me through the course of this work; many ideas
came out of our regular meetings. My other committee members -- Vilma Mesa, Alex Perry, and David Speyer -- reviewed my work and gave useful comments. I am also thankful to fellow PhD students, including Nawaz Sultani, Qiusheng Zhao, Suchitra Pande, Sayantan Khan, and Malavika Mukundan, for helpful conversations.

\section{Admissible covers}
\label{section:admissible}

Having defined our invariants of intersect in the previous section, we now turn
to a
known method in the literature that we will use
to recursively compute invariants $W_{g,\sigma,\nu}$. Cela and Lian use this method to prove Proposition 6 in \cite{Generalized}, relying on a result from an earlier paper of Lian's \cite{Lian}. We begin by motivating the method, and then illustrate the use of the method as adapted to the specific context of this paper. This adaptation takes the form of Theorem~\ref{thm:admissible}, which we will use repeatedly in later sections. We will also contribute an observation that facilitates easier use of Theorem~\ref{thm:admissible}, stated as Lemma~\ref{lem:branchpoints}

Recall that $W_{g,\sigma,\nu}$ is the degree of the map $\pi_{\nu}:\Hb_{g,\sigma}\to\mb_{g,m}$
where $\sum\ell(\nu_i)=m$; our enumerative interpretation is that, for a fixed general $m$-pointed genus $g$ curve $(C,p_1,\dots,p_m)$, $W_{g,\sigma,\nu}$ counts the number of Hurwitz covers of the correct form with source $(C,p_1,\dots,p_m)$. One can hope (and we will confirm) that the same is roughly true for
a specific $m$-pointed genus $g$ {\it stable} curve lying in a suitable boundary divisor of $\mb_{g,m}$.

We first begin with a rough illustration of what we intend to do.

\begin{eg}
  \label{eg:fourtorsion}
  Suppose we are interested in calculating the genus $1$ invariant $W_{(3),(2,1),(2,1),(2,1)}(2,1)$, i.e.\ the degree of the map
  \[
  \pi:\Hb_{1,\sigma}\to\mb_{1,2}
  \]
  
  Let $A$ be a stable graph consisting of a genus $1$ vertex connected to a genus $0$ vertex, with two legs on the
  genus $0$ vertex:

    \scalebox{\BRANCHEDFACTOR}
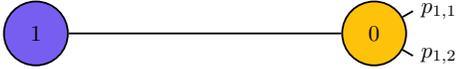
{
  \begin{tikzpicture}[thick,amat/.style={matrix of nodes,nodes in empty cells,
  row sep=2.5em,rounded corners,
  nodes={draw,solid,circle,minimum size=1.0cm}},
  dmat/.style={matrix of nodes,nodes in empty cells,row sep=2.5em,nodes={minimum size=1.0cm},draw=myred},
  fsnode/.style={fill=myblue},
  ssnode/.style={fill=mygreen}]

  \matrix[amat,nodes=fsnode] (mat1) {$1$\\};

 \matrix[amat,right=4cm of mat1,nodes=ssnode] (mat2) {$0$\\};

 \draw  (mat1-1-1) edge (mat2-1-1);
 
 \draw  (mat2-1-1) -- +(30:.7) node[anchor=west] {$p_{1,1}$}
 (mat2-1-1) -- +(330:.7) node[anchor=west] {$p_{1,2}$};

              \end{tikzpicture}}

  Intuitively we want to count covers from a generic stable curve with dual graph $A$ to a genus $0$
  curve, such that $p_{1,1}$ and $p_{1,2}$ are in the same fiber (with ramification indices $2$ and $1$),
  and there is marked ramification elsewhere with ramification profiles $(3),(2,1),(2,1),(2,1)$.
  Remembering the marked points, the source should be a stable curve in $\mb_{1,9}$ with dual graph $\hat A$, where $\hat A$
  becomes $A$ after applying a forgetful map.

  One possible graph $\hat A_1$, obtained by putting $p_{5,1}$ and $p_{5,2}$ on the genus $0$ component and other
  marked points on the genus $1$ component, is the source of a type of admissible cover in $\Hb_{1,\sigma}$:

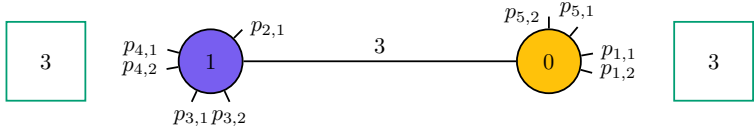
    \scalebox{\BRANCHEDFACTOR}{
                  \begin{tikzpicture}[thick,amat/.style={matrix of nodes,nodes in empty cells,
  row sep=2.5em,rounded corners,
  nodes={draw,solid,circle,minimum size=1.0cm}},
  dmat/.style={matrix of nodes,nodes in empty cells,row sep=2.5em,nodes={minimum size=1.0cm},draw=myred},
  fsnode/.style={fill=myblue},
  ssnode/.style={fill=mygreen}]

  \matrix[amat,nodes=fsnode] (mat1) {$1$\\};

    \matrix[dmat,left=1.3cm of mat1] (degrees1) {$3$\\};

 \matrix[amat,right=4cm of mat1,nodes=ssnode] (mat2) {$0$\\};

 \matrix[dmat,right=1.3cm of mat2] (degrees2) {$3$\\};

 \draw  (mat1-1-1) edge["$3$"] (mat2-1-1);

 \draw (mat1-1-1) -- +(45:.7) node[anchor=west] {$p_{2,1}$}
 (mat1-1-1) -- +(245:.7) node[anchor=north] {$p_{3,1}$}
 (mat1-1-1) -- +(295:.7) node[anchor=north] {$p_{3,2}$}
 (mat1-1-1) -- +(165:.7) node[anchor=east] {$p_{4,1}$}
 (mat1-1-1) -- +(190:.7) node[anchor=east] {$p_{4,2}$};

 \draw  (mat2-1-1) -- +(10:.7) node[anchor=west] {$p_{1,1}$}
 (mat2-1-1) -- +(345:.7) node[anchor=west] {$p_{1,2}$}
 (mat2-1-1) -- +(50:.7) node[anchor=south] {$p_{5,1}$}
 (mat2-1-1) -- +(90:.7) node[anchor=east] {$p_{5,2}$};

                  \end{tikzpicture}}

                  Another possible graph $\hat A_2$, obtained by putting $p_{2,1}$ on the genus $0$ component
                  and splitting the other marked points across a genus $1$ and a new genus $0$ component,
                  is the source of another type of admissible cover in $\Hb_{1,\sigma}$:

                  \scalebox{\BRANCHEDFACTOR}{
                  \begin{tikzpicture}[thick,amat/.style={matrix of nodes,nodes in empty cells,
  row sep=2.5em,rounded corners,
  nodes={draw,solid,circle,minimum size=1.0cm}},
  dmat/.style={matrix of nodes,nodes in empty cells,row sep=2.5em,nodes={minimum size=1.0cm},draw=myred},
  fsnode/.style={fill=myblue},
  ssnode/.style={fill=mygreen}]

                                      \matrix[amat,nodes=fsnode] (mat1) {$1$\\
                                      $0$\\};

  \matrix[dmat,left=1cm of mat1] (degrees1) {$2$\\
  $1$\\};

 \matrix[amat,right=4cm of mat1,nodes=ssnode] (mat2) {$0$\\};

 \matrix[dmat,right=1cm of mat2] (degrees2) {$3$\\};

 \draw  (mat1-1-1) edge["$2$"] (mat2-1-1)
 (mat1-2-1) edge["$1$"] (mat2-1-1);

 \draw (mat1-1-1) -- +(35:.7) node[anchor=west] {$p_{5,1}$}
 (mat1-2-1) -- +(35:.7) node[anchor=west] {$p_{5,2}$}
 (mat1-1-1) -- +(220:.7) node[anchor=north] {$p_{3,1}$}
 (mat1-2-1) -- +(220:.7) node[anchor=north] {$p_{3,2}$}
 (mat1-1-1) -- +(165:.7) node[anchor=east] {$p_{4,1}$}
 (mat1-2-1) -- +(165:.7) node[anchor=east] {$p_{4,2}$};

 \draw  (mat2-1-1) -- +(10:.7) node[anchor=west] {$p_{1,1}$}
 (mat2-1-1) -- +(345:.7) node[anchor=west] {$p_{1,2}$}
 (mat2-1-1) -- +(50:.7) node[anchor=south] {$p_{2,1}$};

                \end{tikzpicture}}

                                    We hope that $W_{(3),(2,1),(2,1),(2,1)}(2,1)$ can be expressed
                                    as a sum of terms corresponding to these kinds of
                                    admissible covers:

                                    There are three graphs similar to $\hat A_1$, corresponding to a choice of $i\in\{3,4,5\}$ for which $p_{i,1}$ and $p_{i,2}$ appear on the genus $0$ component. The contribution from each graph should be a product
                                    \[
                                    W_{(3),(2,1),(2,1)}(3)\cdot \overline H((3),(2,1),(2,1))=W_{(3),(2,1),(2,1)}(3)\cdot 1=2N_{(3),(2,1),(2,1)}(3)=2\cdot 8=16
                                    \]
                                    There is only a single graph like $\hat A_2$,
                                    and its contribution is
                                    \[
                                    W_{(2),(2),(2)}(2)\cdot \overline H((3),(2,1),(2,1))=3!\cdot N_{(2),(2),(2)}(2)\cdot 1=6
                                    \]
                                    Therefore our expectation is that
                                    \[
                                    W_{(3),(2,1),(2,1),(2,1)}(2,1)=3\cdot 16+6=54\implies N_{(3),(2,1),(2,1),(2,1)}(2,1)=\frac{54}{3!}=9.
                                    \]
  \end{eg}

                \subsection{Setup}

                We now turn to a general explanation of this technique loosely applied in the example above, and in following \cite{Generalized} we will write an explicit application of the technique to our specific invariants.

Let $A$ be a stable graph
corresponding to a boundary stratum in $\mb_{g,m}$.
Following Proposition 3.2 of \cite{Lian}, we have a commutative diagram
\[\begin{tikzcd}
\amalg \mathcal H_{(\Gamma,\Gamma')} & \Hb_{g,\sigma} \\
\amalg \mb_{\hat A} & \mb_{g,N} \\
	\mb_A & \mb_{g,m}
	\arrow[from=1-1, to=1-2]
	\arrow[from=1-1, to=2-1, "\amalg \phi_{(\Gamma,\Gamma')}"]
        \arrow[from=2-1, to=3-1]
	\arrow[from=1-2, to=2-2]
	\arrow[from=2-1, to=2-2]
        \arrow[from=2-2, to=3-2]
        \arrow[from=3-1, to=3-2, "\xi_A"]
        \arrow[from=1-2, to=3-2, bend left=52, "\pi_{\nu}"]
        \arrow[from=1-1, to=3-1, bend right=52, "\phi"]
\end{tikzcd}\]

Here the index set of the top left object
consists of pairs $(\Gamma,\Gamma')$
where $\Gamma$ and $\Gamma'$ are stable graphs (the latter being genus $0$), $\Gamma\to\Gamma'$ is an admissible cover of stable graphs, $\Gamma$
comes with a chosen $\hat A$-structure (i.e., a specialization to $\hat A$),
and $\hat A$ is any stable graph reducing to $A$ under the forgetful map.
The index set of $\amalg\overline{\mathcal M}_{\hat A}$ consists of
all stable graphs reducing to $A$ under the forgetful map.

The outside square is Cartesian on the level of closed points. By
Proposition 3.3 of \cite{Lian}, if $\mathcal H_{(\Gamma,\Gamma')}$ is of
expected dimension, its contribution to $\xi_A^*(\pi_{\nu})_*([\Hb_{g,\sigma}])$
is a nonzero multiple of $\phi_*([\mathcal H_{(\Gamma,\Gamma')}])$ on $\overline{\mathcal M}_A$, and in fact the coefficient is given by the product of labels of edges contracted in the $A$-structure of $\Gamma$ divided by the size of the automorphism group.

\begin{thm}
  \label{thm:admissible}
  If $A$ is a stable graph corresponding to a boundary divisor in $\mb_{g,m}$, then
  \[
  \frac 1{|\Aut(A)|}W_{g,\sigma,\nu}=\sum_{\hat A\rightsquigarrow A}\sum_{(\Gamma,\Gamma')}\frac 1{|\Aut(\Gamma)|}c_{\Gamma}W_{(\Gamma,\Gamma')}
  \]
  where $\hat A$ becomes $A$ after forgetting points, $\Gamma'$ is a stable genus $0$ graph with one edge, $\Gamma$ is a stable graph with a chosen $\hat A$-structure and an admissible
  cover to $\Gamma'$, $c_{\Gamma}$ is the
  product of labels of contracted edges in the $\hat A$-structure, and
  $W_{(\Gamma,\Gamma')}$ is a product of terms $W_{g',\sigma',\nu'}$ corresponding to each vertex of $\Gamma$.
\end{thm}

\begin{proof}
  Note that
  \[
  \xi_A^*(\pi_{\nu})_*([\mathcal H_{g,\sigma,\nu}]) =
  \xi_A^*(W_{g,\sigma,\nu}[\mb_{g,m}])=
  W_{g,\sigma,\nu}[\mb_A]
  \]
  Therefore, $W_{g,\sigma,\nu}$ is a sum of
  terms $\frac 1{|\Aut(\Gamma)|}c_{\Gamma}W_{(\Gamma,\Gamma')}$, for
  any $(\Gamma,\Gamma')$ such that
  $\mathcal H_{(\Gamma,\Gamma')}\to\mb_A$
  is dominant.

  We have a natural map $\mathcal H_{(\Gamma,\Gamma')}\to \mb_{\Gamma'}$ remembering the target, which is finite by Hurwitz theory. Therefore,
  \[
  \dim(\mathcal H_{(\Gamma,\Gamma')})=\dim(\mb_{\Gamma'})
  \]
  For $\mathcal H_{(\Gamma,\Gamma')}\to\mb_A$ to be
  dominant means therefore that
  \[
  \dim(\mb_{\Gamma'})=\dim(\mb_A)=
  \dim(\mb_{g,m})-1=\dim(\mb_{0,n})-1
  \]
  or equivalently that $\Gamma'$ has a single edge.
\end{proof}

\subsection{Genus reduction}

We can apply this technique to compute invariants $W_{\sigma_1,\dots,\sigma_n}(d)$ by reduction to
genus $0$ invariants, and illustrate with the following example.

\begin{eg}
  \label{eg:genusreduction}
  Suppose we are interested in calculating $W_{(4),(2,1,1),(2,1,1)}(4)$, the number of degree $4$ covers from a fixed elliptic curve $(E,p_1)$ to $\P^1$ with full ramification over $p_1$ and another point, with ramification points labeled.
  
  Let $A$ be a stable graph consisting of a single genus $0$ vertex with
  a loop and a leg labeled $p_{1,1}$; and let $\hat A$ be a graph with
  additional legs corresponding to the other fibers of ramification.

  While no such admissible cover
  can have actual source of the form described by $\hat A$ (due to the loop), there are in fact
  covers whose source {\it stabilizes} to the kind we want.
  Consider the following picture:

    \scalebox{\BRANCHEDFACTOR}
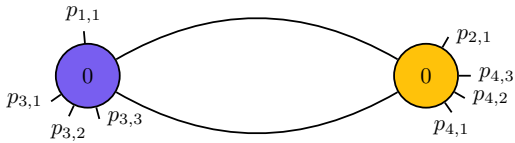
{
              \begin{tikzpicture}[thick,amat/.style={matrix of nodes,nodes in empty cells,
  row sep=2.5em,rounded corners,
  nodes={draw,solid,circle,minimum size=1.0cm}},
  dmat/.style={matrix of nodes,nodes in empty cells,row sep=2.5em,nodes={minimum size=1.0cm},draw=myred},
  fsnode/.style={fill=myblue},
  ssnode/.style={fill=mygreen}]

  \matrix[amat,nodes=fsnode] (mat1) {$0$\\};

 \matrix[amat,right=4cm of mat1,nodes=ssnode] (mat2) {$0$\\};

 \draw  (mat1-1-1) edge[bend left] (mat2-1-1)
 (mat1-1-1) edge[bend right] (mat2-1-1);

 \draw (mat1-1-1) -- +(95:.7) node[anchor=south] {$p_{1,1}$}
 (mat1-1-1) -- +(215:.7) node[anchor=east] {$p_{3,1}$}
 (mat1-1-1) -- +(245:.7) node[anchor=north] {$p_{3,2}$}
 (mat1-1-1) -- +(285:.7) node[anchor=west] {$p_{3,3}$};

 \draw  (mat2-1-1) -- +(60:.7) node[anchor=west] {$p_{2,1}$}
 (mat2-1-1) -- +(305:.7) node[anchor=north] {$p_{4,1}$}
 (mat2-1-1) -- +(330:.7) node[anchor=west] {$p_{4,2}$}
 (mat2-1-1) -- +(360:.7) node[anchor=west] {$p_{4,3}$};

              \end{tikzpicture}}

  This is a stable graph $\Gamma$ that specializes to
  $\hat A$ in two ways, by contracting either of
  the two edges; each corresponds to a choice of $\hat A$-structure
  as defined in \cite{Schmitt}.
  In the sense of Definition 2.10 from \cite{Lian}, there is an admissible cover of stable graphs $\Gamma\to\Gamma'$ depicted as follows where $a+b=4$:

    \scalebox{\BRANCHEDFACTOR}
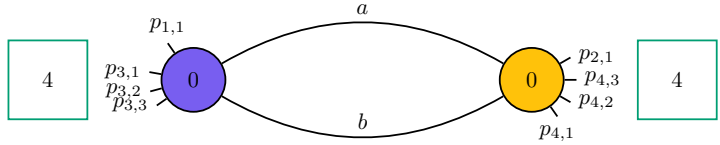
{
                \begin{tikzpicture}[thick,amat/.style={matrix of nodes,nodes in empty cells,
  row sep=2.5em,rounded corners,
  nodes={draw,solid,circle,minimum size=1.0cm}},
  dmat/.style={matrix of nodes,nodes in empty cells,row sep=2.5em,nodes={minimum size=1.0cm},draw=myred},
  fsnode/.style={fill=myblue},
  ssnode/.style={fill=mygreen}]

  \matrix[amat,nodes=fsnode] (mat1) {$0$\\};

    \matrix[dmat,left=1cm of mat1] (degrees1) {$4$\\};

 \matrix[amat,right=4cm of mat1,nodes=ssnode] (mat2) {$0$\\};

 \matrix[dmat,right=1cm of mat2] (degrees2) {$4$\\};

 \draw  (mat1-1-1) edge[bend left,"$a$"] (mat2-1-1)
 (mat1-1-1) edge[bend right,"$b$"] (mat2-1-1);

 \draw (mat1-1-1) -- +(125:.7) node[anchor=south] {$p_{1,1}$}
 (mat1-1-1) -- +(170:.7) node[anchor=east] {$p_{3,1}$}
 (mat1-1-1) -- +(195:.7) node[anchor=east] {$p_{3,2}$}
 (mat1-1-1) -- +(215:.7) node[anchor=east] {$p_{3,3}$};

 \draw  (mat2-1-1) -- +(30:.7) node[anchor=west] {$p_{2,1}$}
  (mat2-1-1) -- +(305:.7) node[anchor=north] {$p_{4,1}$}
 (mat2-1-1) -- +(330:.7) node[anchor=west] {$p_{4,2}$}
 (mat2-1-1) -- +(360:.7) node[anchor=west] {$p_{4,3}$};

                \end{tikzpicture}}

    The contribution of such $(\Gamma,\Gamma')$, considering both possible $\hat A$-structures (the sum of whose
    coefficients $c_{\Gamma}$ is $a+b=4$), is
                \[
                \frac 1{|\Aut(\Gamma)|}\cdot 4\cdot W_{(\Gamma,\Gamma')}=\frac 1{|\Aut(a,b)|}\cdot 4\cdot \overline H((4),(a,b),(2,1,1))\cdot \overline H((4),(a,b),(2,1,1))
                \]
                When $(a,b)=(2,2)$ this is
                \[
                \frac 12\cdot 4\cdot 2\cdot 2=8
                \]
                and when $(a,b)=(3,1)$ this is
                \[
                1\cdot 4\cdot 2\cdot 2=16.
                \]
                There are similar choices of $(\Gamma,\Gamma')$ where the components of $q_3$ and
                $q_4$ are swapped,
                yielding a cumulative sum of $2\cdot (8+16)=48$.

                The other type of cover $\Gamma\to\Gamma'$ to consider can be depicted as follows:

                  \scalebox{\BRANCHEDFACTOR}
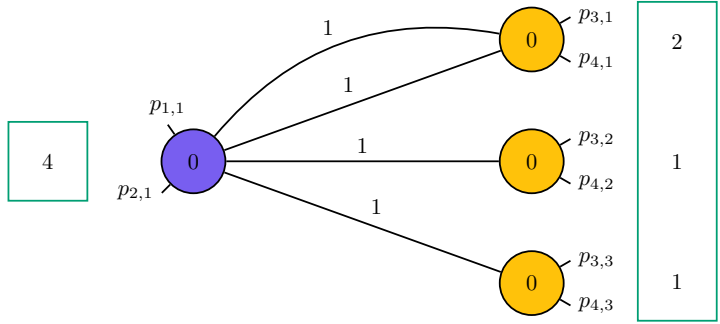
{
                                \begin{tikzpicture}[thick,amat/.style={matrix of nodes,nodes in empty cells,
  row sep=2.5em,rounded corners,
  nodes={draw,solid,circle,minimum size=1.0cm}},
  dmat/.style={matrix of nodes,nodes in empty cells,row sep=2.5em,nodes={minimum size=1.0cm},draw=myred},
  fsnode/.style={fill=myblue},
  ssnode/.style={fill=mygreen}]

  \matrix[amat,nodes=fsnode] (mat1) {$0$\\};

    \matrix[dmat,left=1cm of mat1] (degrees1) {$4$\\};

    \matrix[amat,right=4cm of mat1,nodes=ssnode] (mat2) {$0$\\
      $0$\\
    $0$\\};

 \matrix[dmat,right=1cm of mat2] (degrees2) {$2$\\
   $1$\\
 $1$\\};

 \draw  (mat1-1-1) edge[bend left,"$1$"] (mat2-1-1)
 (mat1-1-1) edge["$1$"] (mat2-1-1)
 (mat1-1-1) edge["$1$"] (mat2-2-1)
 (mat1-1-1) edge["$1$"] (mat2-3-1);

 \draw (mat1-1-1) -- +(125:.7) node[anchor=south] {$p_{1,1}$}
 (mat1-1-1) -- +(225:.7) node[anchor=east] {$p_{2,1}$};

 \draw  (mat2-1-1) -- +(30:.7) node[anchor=west] {$p_{3,1}$}
 (mat2-2-1) -- +(30:.7) node[anchor=west] {$p_{3,2}$}
 (mat2-3-1) -- +(30:.7) node[anchor=west] {$p_{3,3}$}
 (mat2-1-1) -- +(330:.7) node[anchor=west] {$p_{4,1}$}
 (mat2-2-1) -- +(330:.7) node[anchor=west] {$p_{4,2}$}
 (mat2-3-1) -- +(330:.7) node[anchor=west] {$p_{4,3}$};

                                \end{tikzpicture}}

                  There are two possible $\hat A$-structures for $\Gamma$ corresponding to which of the top two edges is
                  contracted (alongside the bottom two edges), and the sum of corresponding terms $c_{\Gamma}$ is
                  $1+1=2$.
                                There are two choices for placement of $p_{3,2}$ and $p_{4,2}$
                                and two automorphisms of the graph
                                (swapping the two top edges), so
                                the total contribution
                                is
                                \begin{align*}
                                  2\cdot \frac 1{|\Aut(\Gamma)|}\cdot 2\cdot W_{(\Gamma,\Gamma')}&=2\cdot\frac 12\cdot 2\cdot \overline H((4),(4),(1,1,1,1))\cdot \overline H((2),(2),(1,1)) \\
                                  &=2\cdot 6\cdot 1\\
                                  &=12
                                \end{align*}
                                Combining our results,
                                \[
                                \frac 12W_{(4),(2,1,1),(2,1,1)}(4)=48+12=60\implies W_{(4),(2,1,1),(2,1,1)}(4)=120\implies N_{(4),(2,1,1),(2,1,1)}(4)=\frac{120}{2^3}=15
                                \]
                                This matches our expectation, since $N_{(4),(2,1,1),(2,1,1)}(4)$ counts
                                nontrivial $4$-torsion points of an elliptic curve.
                                
\end{eg}

In general, any invariant $W_{\sigma_1,\dots,\sigma_n}(d)$ can be computed via this method, using the
same graph $A$ in this example and reducing to genus $0$ invariants (which will all be marked
Hurwitz numbers). One can also apply genus reduction for any
invariant $W_{\sigma_1,\dots,\sigma_n}(\mu)$ with $\ell(\mu)\geq 2$, although in
this case we will end up with enumerative counts in genus $0$ that
are more complicated than Hurwitz numbers.

  \subsection{Merging two points}

  In the case where $\ell(\mu)\geq 2$, we can consider the following graph $A$:

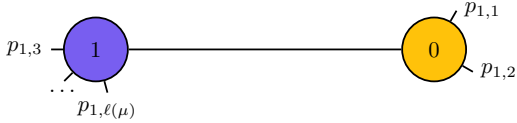
  \scalebox{\BRANCHEDFACTOR}{
  \begin{tikzpicture}[thick,amat/.style={matrix of nodes,nodes in empty cells,
  row sep=2.5em,rounded corners,
  nodes={draw,solid,circle,minimum size=1.0cm}},
  dmat/.style={matrix of nodes,nodes in empty cells,row sep=2.5em,nodes={minimum size=1.0cm},draw=myred},
  fsnode/.style={fill=myblue},
  ssnode/.style={fill=mygreen}]

  \matrix[amat,nodes=fsnode] (mat1) {$1$\\};

 \matrix[amat,right=4cm of mat1,nodes=ssnode] (mat2) {$0$\\};

 \draw  (mat1-1-1) edge (mat2-1-1);

 \draw (mat1-1-1) -- +(180:.7) node[anchor=east] {$p_{1,3}$}
 (mat1-1-1) -- +(225:.7) node[anchor=north] {$\dots$}
 (mat1-1-1) -- +(285:.7) node[anchor=north] {$p_{1,\ell(\mu)}$};

 \draw  (mat2-1-1) -- +(60:.7) node[anchor=west] {$p_{1,1}$}
 (mat2-1-1) -- +(330:.7) node[anchor=west] {$p_{1,2}$};

  \end{tikzpicture}}

                Applying Theorem~\ref{thm:admissible} to this graph, we can compute the value of
                $W_{\sigma_1,\dots,\sigma_n}(\mu)$ by reduction to marked Hurwitz
                numbers and terms of the form
                $W_{\tau_1,\dots,\tau_m}(\mu')$ where $\ell(\mu')<\ell(\mu)$, for instance
                in the first example of this section.
                The base case where $\ell(\mu)=1$ can be solved via genus reduction.
                In principle, this gives us a complete (albeit potentially combinatorially difficult) method to compute any invariant $W_{\sigma_1,\dots,\sigma_n}(\mu)$.

                More generally, we can use the theorem to compute $W_{g,\sigma,\nu}$ when $m=\sum_{i=1}^n\ell(\nu_i)\geq 2$, by letting $A$ be a graph with a genus $g$ vertex attached to a genus $0$ vertex, and two of the
                marked points on the genus $0$ vertex. The following lemma will be useful:

                \begin{lem}
                  \label{lem:branchpoints}
                  Letting $A$ be as above, in Theorem~\ref{thm:admissible}, every admissible cover of graphs $\Gamma\to\Gamma'$ where $\Gamma$ contains a genus $g$ vertex (guaranteed when $g=1$) has precisely $2$ branch points
                  whose fiber does not intersect the genus $g$ vertex.
                \end{lem}

                \begin{proof}
                  Let $\Gamma\to\Gamma'$ be an admissible cover of graphs where $\Gamma$ comes with a $\hat A$-structure and $\hat A$ specializes to $A$, and assume $\Gamma$ has a genus $g$ vertex $v$.
                  Then $\Hb_{(\Gamma,\Gamma')}$ has a factor $\Hb_{v}$ corresponding to the vertex $v$,
                  and there is a finite map $\Hb_{v}\to\mb_{g,m-1}$ which recalls
                  the vertex $v$ along with the $m-2$ fixed points on the genus $g$ vertex of $A$ and
                  the point corresponding to the node of $A$. Thus,
                  \[
                  \dim(\Hb_v)=\dim(\mb_{g,m-1})=3g-3+m-1
                  \]
                  By Lemma~\ref{lem:dim}, $\dim(\Hb_v)=n-4$. There is also a finite morphism
                  $\Hb_v\to\mb_{0,k+1}$ where $k$ is the number of proper branch points from this
                  component (and thus $k+1$ is the number of branch points including the node), so
                  \[
                  k-2=n-4\implies k=n-2
                  \]
                  and therefore the number of branch points not arising from the genus $g$ vertex is $n-(n-2)=2$.
                  \end{proof}

                \section{All twos case}
                \label{section:twos}

In this section we study the invariants $N_{(2^k),(2^k)}(\mu)$ defined for $d=2k$ even, arriving at a proof of Theorem~\ref{thm:twos}.

First, we establish the base case $\mu=(d)$, with two distinct proofs (one by Hurwitz theory, the other by genus reduction). We then prove the general result by induction with
our most powerful technique: establishing a recursion between these invariants by use of Theorem~\ref{thm:admissible}.

\subsection{Motivation}

The Gromov-Witten invariants of an elliptic curve are well understood (\cite{Completed}). However, a natural generalization
remains open: let $C$ be the quotient of an elliptic curve under the action $x\mapsto -x$.
This action has four fixed points (the two-torsion points in the group law), and so as an orbifold
$C$ can be identified as $\P^1$ with four points of stabilizer $\Z/2$. This orbifold $C$ is known as the
{\it pillowcase orbifold}. Maps $E\to C$ correspond to maps $E\to\P^1$ with ``even'' ramification over
the four special points; one can hope to use the degeneration formula (as developed in \cite{Degeneration}) to find
relative Gromov-Witten invariants of $C$, splitting $C$ into a union
of $\P^1s$ with two $\Z/2$-stabilizer points each and a node, and $N_{(2^k),(2^k)}(\mu)$ is a relative Gromov-Witten invariant
for each of these two special $\P^1s$.
In general, one can hope to understand the sequence of classes $\langle \mu\rangle_d^{C}\in H^{g-1}(\overline{\mathcal M}_{g,\ell(\mu)})$.

\subsection{$\mu=(d)$}

We will show that $N_{(2^k),(2^k)}(d)=3$
using two methods: first, by categorizing all such maps explicitly using Hurwitz theory and counting them, and second by performing a genus reduction.

\subsubsection{Combinatorial solution}
\label{section:comb}

\begin{figure}[h]
  \caption{Factored Hurwitz cover with even ramification, in the case where $k$ is odd}
  \centering
\begin{tikzpicture}
  \draw[domain=-6:7.5,samples=50,color=gray!50,thick] plot (\x, \x^3/64 - \x/2); 
  \draw[domain=-6:7.5,samples=50,color=gray!50,thick] plot (\x, -4); 
  \draw[domain=-6:7.5,samples=50,color=gray!50,thick] plot (\x, -8); 

  \foreach \p/\plabel/\qlabel in {-5/$p$/$q_1$, -1//$q_2$, 3//$q_3$, 7//$q_4$} {
    \filldraw[black] (\p, \p^3/64-\p/2) circle (0.6mm) node[above] {\plabel};
    \draw[->] (\p, \p^3/64-\p/2-.4) to["2"] (\p, -3.5);
    \filldraw[black] (\p, -4) circle (0.6mm) node[below] {\qlabel};
  }

  \draw[->] (-5, -4.5) to["$k$"] (-5, -7.5);

  \foreach \x/\l in {-4/$2$, -3/$\dots$, -2/$2$, -1/$1$} {
    \draw[->] (\x, -4.5) to["\l"] (-3/2+\x/4, -7.5);
  }

  \foreach \x/\l in {0/$2$, 1/$\dots$, 2/$2$, 3/$1$} {
    \draw[->] (\x, -4.5) to["\l"] (3/2+\x/4, -7.5);
  }

  \foreach \x/\l in {4/$1$, 5/$1$, 6/$\dots$, 7/$1$} {
    \draw[->] (\x, -4.5) to["\l"] (9/2+\x/4, -7.5);
  }

  \foreach \x/\l in {-5/$0$, -2/$1$, 2/$\infty$, 6/$w$} {
    \filldraw[black] (\x, -8) circle (0.6mm) node[below] {\l};
    }

\end{tikzpicture}

\label{fig:comb1}
\end{figure}
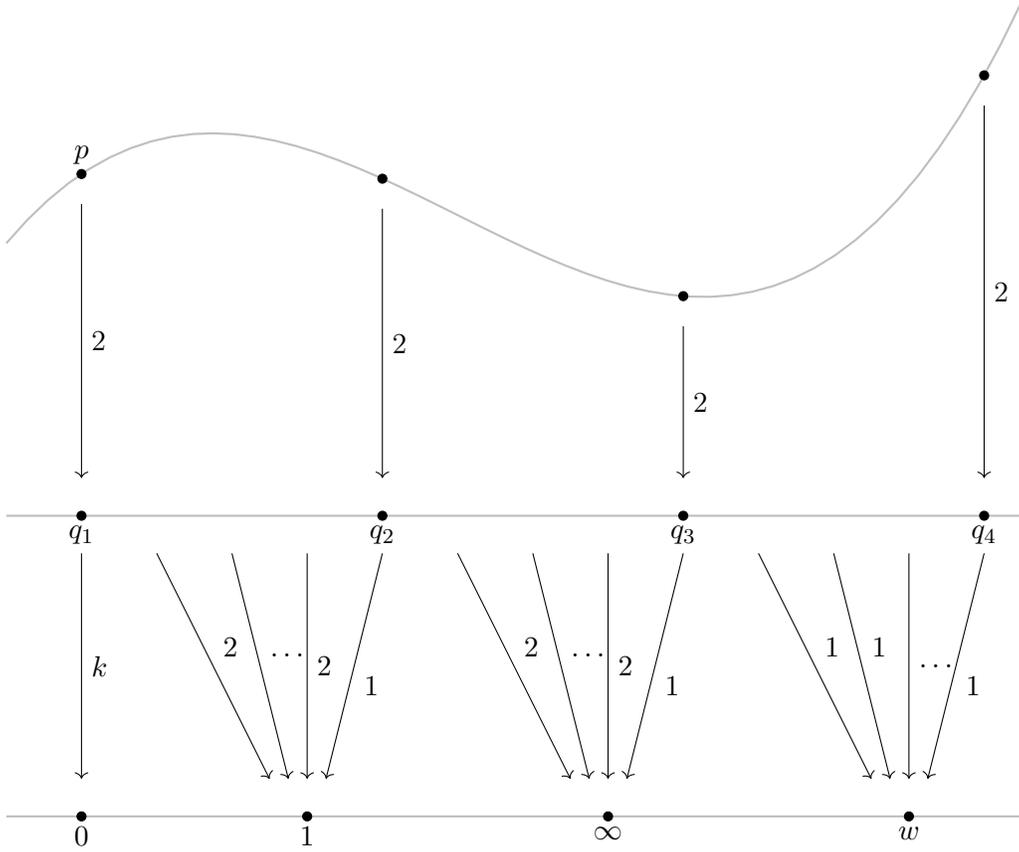

\begin{figure}[h]
  \caption{Factored Hurwitz cover with even ramification, in the case where $k$ is even}
  \centering
\begin{tikzpicture}
  \draw[domain=-6:7.5,samples=50,color=gray!50,thick] plot (\x, \x^3/64 - \x/2); 
  \draw[domain=-6:7.5,samples=50,color=gray!50,thick] plot (\x, -4); 
  \draw[domain=-6:7.5,samples=50,color=gray!50,thick] plot (\x, -8); 

  \foreach \p/\plabel/\qlabel in {-5/$p$/$q_1$, -4//$q_2$, -3//$q_3$, 5//$q_4$} {
    \filldraw[black] (\p, \p^3/64-\p/2) circle (0.6mm) node[above] {\plabel};
    \draw[->] (\p, \p^3/64-\p/2-.4) to["2"] (\p, -3.5);
    \filldraw[black] (\p, -4) circle (0.6mm) node[below] {\qlabel};
  }

  \draw[->] (-5, -4.5) to["$k$"] (-5, -7.5);

  \foreach \x/\l in {-4/$1$, -3/$1$, -2/$2$, -1/$\dots$, 0/$2$} {
    \draw[->] (\x, -4.5) to["\l"] (-3/2+\x/4, -7.5);
  }

  \foreach \x/\l in {1/$2$, 2/$\dots$, 3/$2$} {
    \draw[->] (\x, -4.5) to["\l"] (3/2+\x/4, -7.5);
  }

  \foreach \x/\l in {5/$1$, 6/$\dots$, 7/$1$} {
    \draw[->] (\x, -4.5) to["\l"] (9/2+\x/4, -7.5);
  }

  \foreach \x/\l in {-5/$0$, -2/$1$, 2/$\infty$, 6/$w$} {
    \filldraw[black] (\x, -8) circle (0.6mm) node[below] {\l};
    }

\end{tikzpicture}

\label{fig:comb2}
\end{figure}

\begin{thm}
  \label{thm:combinatorial}
  Let $f:E\to\P^1$ be of degree $d=2k$ such that $f(p)=0$ with full ramification,
  and $f$ has ramification profiles $(2^k)$ over $1$ and $\infty$. Then $f$ factors as
  $E\to\P^1\to\P^1$ where the first map is degree $2$ and the second is degree $k$.
\end{thm}
\begin{proof}
  Figures \ref{fig:comb1} and \ref{fig:comb2} illustrate how such maps should factor. We prove the claim by fixing branch points $0,1,\infty,w$ in $\P^1$, counting the (orbifold) number of factored
  maps over these branch points as well as the maps in Figure \ref{fig:twos}, and showing that these are equal.

\begin{figure}[h]
  \caption{Hurwitz covers with even ramification}
  \centering
\begin{tikzpicture}
  \draw[domain=-6:7.5,samples=50,color=gray!50,thick] plot (\x, \x^3/64 - \x/2); 
  \draw[domain=-6:7.5,samples=50,color=gray!50,thick] plot (\x, -4); 

  \filldraw[black] (-5, .55) circle (0.6mm) node[above] {$p$};

  \draw[->] (-5, 0) to["$d$"] (-5, -3.5);

  \foreach \x/\l in {-3/$2$, -2/$\dots$, -1/$2$} {
    \draw[->] (\x, \x^3/64-\x/2-.4) to["\l"] (-3/2+\x/4, -3.5);
  }

  \foreach \x/\l in {1/$2$, 2/$\dots$, 3/$2$} {
    \draw[->] (\x, \x^3/64-\x/2-.4) to["\l"] (3/2+\x/4, -3.5);
  }

  \foreach \x/\l in {4/$2$, 5/$1$, 6/$\dots$, 7.2/$1$} {
    \draw[->] (\x, \x^3/64-\x/2-.4) to["\l"] (9/2+\x/4, -3.5);
  }

  \foreach \x/\l in {-5/$0$, -2/$1$, 2/$\infty$, 6/$w$} {
    \filldraw[black] (\x, -4) circle (0.6mm) node[below] {\l};
    }

\end{tikzpicture}

\label{fig:twos}
\end{figure}
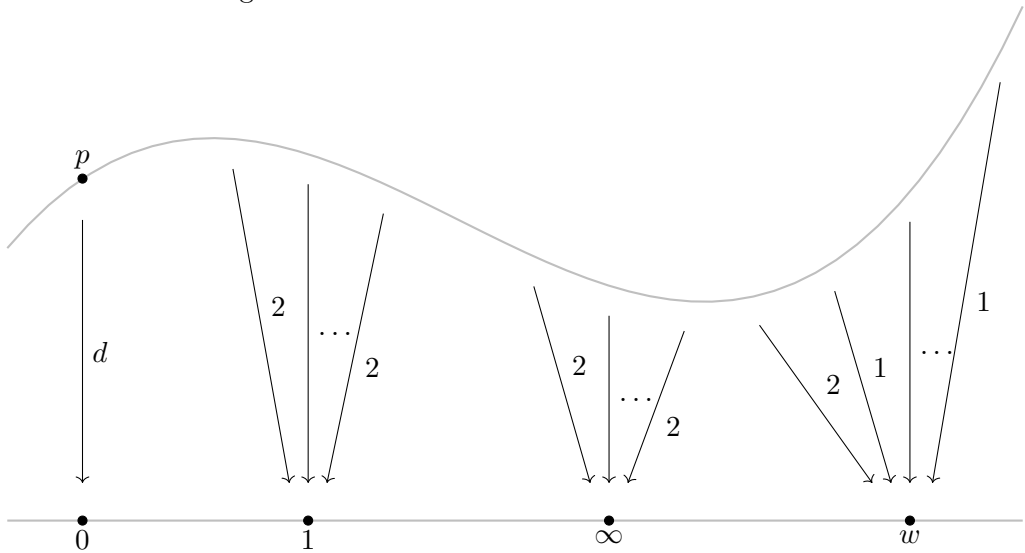

  The latter count is given by
  \[
  H((d),(2^k),(2^k),(2,1,\dots,1))
  \]
  which is $k/2$ by Lemma~\ref{lem:twoscomplete}. For the former count, we
  first consider maps $\P^1\to\P^1$  (with $0,1,\infty,w$ fixed)
  described in Figure \ref{fig:comb1} or \ref{fig:comb2}, depending on the parity of $k$:

  \begin{itemize}
  \item If $k$ is odd, the orbifold number of such maps is $H((k),(2^{(k-1)/2},1),(2^{(k-1)/2},1))$ which
    equals $1$ by Lemma~\ref{lem:twosodd}. There are $k$ choices for $q_1,q_2,q_3,q_4$
    since $q_2$ and $q_3$ are determined and there are $k$ choices for $q_4$.
  \item If $k$ is even, the orbifold number of such maps is $H((k),(2^{k/2-1},1,1),(2^{k/2}))$ which equals $\frac 12$ by Lemma~\ref{lem:twoseven}. There are $2k$ choices for $q_1,q_2,q_3,q_4$
    since there are two choices for labeling $q_2$ and $q_3$, and there are $k$ choices for
    $q_4$.
  \end{itemize}
  In either case, the orbifold number of such covers with $q_1,q_2,q_3,q_4$ labeled is $k$.

  Once $q_1,q_2,q_3,q_4$ are fixed, the count for the upstairs map $E\to\P^1$ is
  $H((2),(2),(2),(2))=\frac 12$. Thus, the orbifold number of factored maps lying over $0,1,\infty,w$ is $\frac 12\cdot k$, proving the claim.
  
\end{proof}

Now given this result, we can compute $N_{(2^k),(2^k)}(d)$ by counting
the number of factored maps described in Figures \ref{fig:comb1} and \ref{fig:comb2} for a fixed $(E,p)$.
The map $(E,p)\to(\P^1,q_1)$ is the elliptic map to $\P^1$ and therefore unique up
to automorphisms of $\P^1$.
There are $3$ choices for the point
$q_4$, and we consider each parity case:
\begin{itemize}
\item If $k$ is odd, then the orbifold count of maps
  $\P^1\to\P^1$ with ramification profiles $(k)$ over $0$,
  $(2^{(k-1)/2},1)$ over $1$, and $(2^{(k-1)/2},1)$ over $\infty$
  is $1$. There is a unique way to assign $q_2$ lying over $1$
  and $q_3$ lying over $\infty$.
  \item
  If $k$ is even, then the orbifold count of maps $\P^1\to\P^1$ with
  ramification profiles $(k)$ over $0$, $(2^{k/2-1},1,1)$ over $1$, and
  $(2^{k/2})$ over $\infty$ is $\frac 12$. There are two choices for
  labeling $q_2$ and $q_3$, so with $q_1,q_2,q_3$ fixed, the number
  of covers is $\frac 12\cdot 2=1$.
\end{itemize}
In either case, $w$ is determined as the image of $q_4$.
Therefore,
\[
N_{(2^k),(2^k)}(d)=3.
\]

\subsubsection{Genus reduction}

We apply the method of Section~\ref{section:admissible} to yield another proof that
$N_{(2^k),(2^k)}(d)=3$:

 Let $A$ be a stable graph consisting of a genus $0$ vertex with a loop and a marked point $p_{1,1}$.
  Let $\hat A$ be the stable graph with additional marked points on the genus $0$ vertex for
  the further ramification.

There are two kinds of
admissible covers $\Gamma\to\Gamma'$ with $\Gamma$ specializing to $\hat A$. The first
is of the form:

  \scalebox{\BRANCHEDFACTOR}{
\begin{tikzpicture}[thick,amat/.style={matrix of nodes,nodes in empty cells,
  row sep=2.5em,rounded corners,
  nodes={draw,solid,circle,minimum size=1.0cm}},
  dmat/.style={matrix of nodes,nodes in empty cells,row sep=2.5em,nodes={minimum size=1.0cm},draw=myred},
  fsnode/.style={fill=myblue},
  ssnode/.style={fill=mygreen}]

  \matrix[amat,nodes=fsnode] (mat1) {$0$\\};

    \matrix[dmat,left=1cm of mat1] (degrees1) {$d$\\};

 \matrix[amat,right=4cm of mat1,nodes=ssnode] (mat2) {$0$\\};

 \matrix[dmat,right=1cm of mat2] (degrees2) {$d$\\};

 \draw  (mat1-1-1) edge["$a$",bend left] (mat2-1-1)
 (mat1-1-1) edge["$b$"] (mat2-1-1);

 \draw (mat1-1-1) -- +(95:.7) node[anchor=south] {$p_{1,1}$}
 (mat1-1-1) -- +(245:.7) node[anchor=north] {$(2,1^{d-2})$};

 \draw  (mat2-1-1) -- +(30:.7) node[anchor=west] {$(2^k)$}
 (mat2-1-1) -- +(330:.7) node[anchor=west] {$(2^k)$};

            \end{tikzpicture}}

            There are two $\hat A$-structures, corresponding to a choice
            of edge to contract, and so the sum of $c_{\Gamma}$ over these $\hat A$-structures is $a+b=d$. The right-hand count here under Theorem~\ref{thm:admissible} is
            \begin{align*}
              \overline H((d),(a,b),(2,1^{d-2}))\cdot \overline H((2^k),(2^k),(a,b))\cdot d\cdot\frac 1{|\Aut(a,b)|}
            \end{align*}
            By Lemma~\ref{lem:twoscycles} in Appendix~\ref{appendix:hurwitz}, the second term vanishes unless $a=b=k$, in which case
            the count (by Claims~\ref{lem:twoscycles} and \ref{lem:Hurwitz2}) is
            \[
            (1\cdot (d-2)!)\cdot \left(\frac 1d\cdot k!\cdot k!\cdot 2\right)\cdot d\cdot\frac 12
            =(d-2)!\cdot (k!)^2
            \]

            Now the second picture is of the form:

            \scalebox{\BRANCHEDFACTOR}
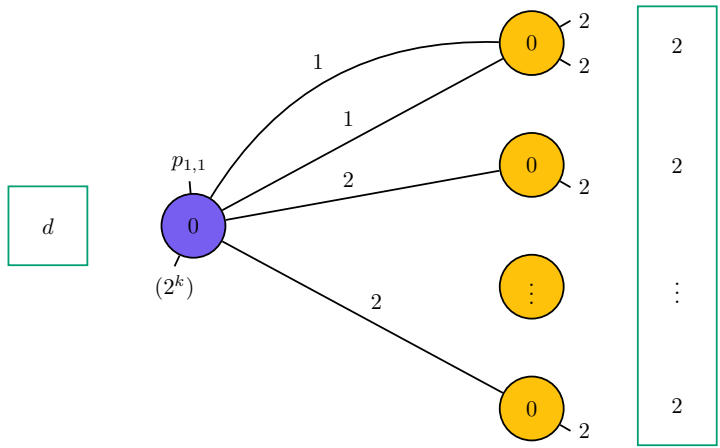
{
              \begin{tikzpicture}[thick,amat/.style={matrix of nodes,nodes in empty cells,
  row sep=2.5em,rounded corners,
  nodes={draw,solid,circle,minimum size=1.0cm}},
  dmat/.style={matrix of nodes,nodes in empty cells,row sep=2.5em,nodes={minimum size=1.0cm},draw=myred},
  fsnode/.style={fill=myblue},
  ssnode/.style={fill=mygreen}]

  \matrix[amat,nodes=fsnode] (mat1) {$0$\\};

    \matrix[dmat,left=1cm of mat1] (degrees1) {$d$\\};

 \matrix[amat,right=4cm of mat1,nodes=ssnode] (mat2) {$0$\\
      $0$\\
      $\vdots$\\
  $0$\\};

 \matrix[dmat,right=1cm of mat2] (degrees2) {$2$\\
   $2$ \\
    \vdots \\
    $2$\\};

 \draw  (mat1-1-1) edge["$1$",bend left] (mat2-1-1)
 (mat1-1-1) edge["$1$"] (mat2-1-1)
 (mat1-1-1) edge["$2$"] (mat2-2-1)
 (mat1-1-1) edge["$2$"] (mat2-4-1);

 \draw (mat1-1-1) -- +(95:.7) node[anchor=south] {$p_{1,1}$}
 (mat1-1-1) -- +(245:.7) node[anchor=north] {$(2^k)$};

 \draw  (mat2-1-1) -- +(30:.7) node[anchor=west] {$2$}
 (mat2-1-1) -- +(330:.7) node[anchor=west] {$2$}
 (mat2-2-1) -- +(330:.7) node[anchor=west] {$2$}
 (mat2-4-1) -- +(330:.7) node[anchor=west] {$2$};

            \end{tikzpicture}}

            With markings remembered, the number of such diagrams is
            \[
            2\cdot k\cdot\binom{d-2}{2}\cdot\binom{d-4}{2}\cdot\dots\cdot\binom 22=2k\cdot\frac{(d-2)!}{2^{k-1}}
            \]
            There are two $\hat A$-structures, each with associated coefficient $2^{k-1}$,
            and so the combined contribution to the right-hand side of Theorem~\ref{thm:admissible} is
            \begin{align*}
              &2k\cdot\frac{(d-2)!}{2^{k-1}}\cdot\overline H((d),(2^k),(2^{k-1},1,1))\cdot \overline H((1,1),(2),(2))\cdot \overline H((2),(2),(1,1))^{k-1}\cdot\frac 12\cdot 2\cdot 2^{k-1} \\
              &=2k\cdot (d-2)!\cdot (k!\cdot (k-1)!) \\
              &=2(d-2)!\cdot (k!)^2
            \end{align*}
            We conclude by adding up and seeing that the right-hand side of Theorem~\ref{thm:admissible} is
            \begin{align*}
              (d-2)!\cdot (k!)^2+ 2(d-2)!\cdot (k!)^2=3(d-2)!\cdot (k!)^2&\implies W_{(2^k),(2^k),(2,1^{d-2})}(d)=6(d-2)!\cdot (k!)^2 \\
              &\implies N_{(2^k),(2^k)}(d)=3.
            \end{align*}

\subsection{Recursion and general solution}

We now aim to find a recursion for the terms $N_{(2^k),(2^k)}(\mu)$. In particular we will show:

\begin{thm}
  If $a>b$, $\mu=(a,b,\mu')$ and $d'=|\mu'|$, then
  \[
  N_{(2^k),(2^k)}(a,b,\mu')=N_{(2^k),(2^k)}(a+b,\mu')+n_{a,b,d'}N_{(2^{k-b}),(2^{k-b})}(a-b,\mu')
  \]
  where \[n_{a,b,d'}=\begin{cases}
  3 & a-b=2,\ d'=0 \\
  6 & a-b=1,\ d'=1 \\
  1 & d-2b>2 \end{cases}
  \]
  and
  \[
  N_{(2^k),(2^k)}(a,a,\mu')=
  N_{(2^k),(2^k)}(2a,\mu')+
  \begin{cases}
    2N_{(2^{k-a}),(2^{k-a})}(\mu') & |\mu'|>2 \\
    3 & |\mu'|=0 \\
    6 & \mu'=(2) \\
    12 & \mu'=(1,1)
    \end{cases}
  \]
  \end{thm}
\begin{proof}
  We will use Theorem~\ref{thm:admissible} to compute $W_{(2^k),(2^k),(2,1^{d-2})^{\ell(\mu)}}(a,b,\mu')$,
  letting $A$ be the graph with a genus $1$ vertex containing
  $p_{1,3},\dots,p_{1,\ell(\mu)}$ connected to a genus $0$ vertex containing $p_{1,1}$ and $p_{1,2}$.
  Any admissible cover of graphs $\Gamma\to\Gamma'$ has precisely two branch points that do
  not arise from the genus $1$ vertex of $\Gamma$, by Lemma~\ref{lem:branchpoints}.

  Suppose the points $p_{1,1},\dots,p_{1,\ell(\mu)}$ appear on the right side of $\Gamma$, i.e.
  their images are not on the same component as the image of the genus $1$ vertex. If
  $p_{1,1}$ and $p_{1,2}$ are on the same component, then we have the picture:

    \scalebox{\BRANCHEDFACTOR}
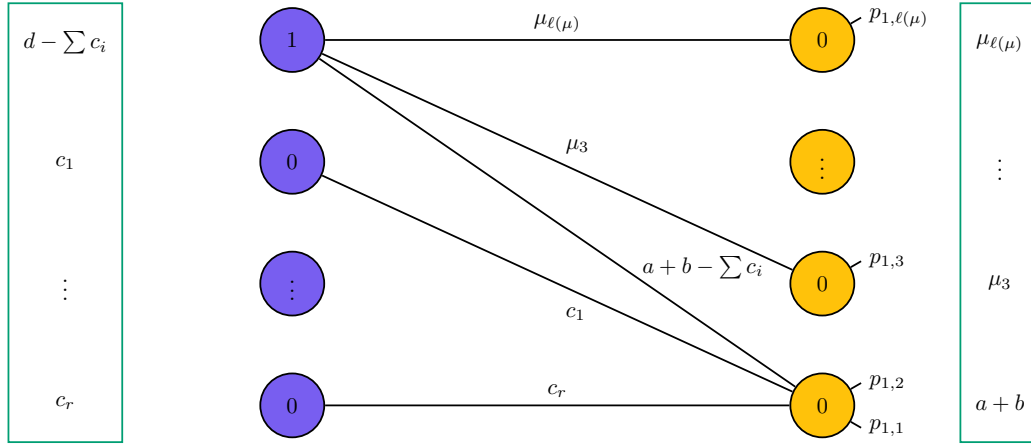
{
  \begin{tikzpicture}[thick,amat/.style={matrix of nodes,nodes in empty cells,
  row sep=2.5em,rounded corners,
  nodes={draw,solid,circle,minimum size=1.0cm}},
  dmat/.style={matrix of nodes,nodes in empty cells,row sep=2.5em,nodes={minimum size=1.0cm},draw=myred},
  fsnode/.style={fill=myblue},
  ssnode/.style={fill=mygreen}]

  \matrix[amat,nodes=fsnode] (mat1) {$1$\\
    $0$\\
    \vdots \\
  $0$\\};

  \matrix[dmat,left=2cm of mat1] (degrees1) {$d-\sum c_i$\\
    $c_1$\\
    \vdots \\
  $c_r$\\};

 \matrix[amat,right=7cm of mat1,nodes=ssnode] (mat2) {$0$\\
   \vdots\\
   $0$ \\
   $0$\\};
 \matrix[dmat,right=1.5cm of mat2] (degrees2) {$\mu_{\ell(\mu)}$\\
   \vdots\\
   $\mu_3$ \\
 $a+b$\\};

 \draw  (mat1-1-1) edge["$\mu_{\ell(\mu)}$"] (mat2-1-1)
 (mat1-1-1) edge["$\mu_{3}$"] (mat2-3-1)
 (mat1-1-1) edge["$a+b-\sum c_i$" shift={(1.2,-1)}] (mat2-4-1)
 (mat1-2-1) edge["$c_1$" shift={(0,-0.7)}] (mat2-4-1)
 (mat1-4-1) edge["$c_r$"] (mat2-4-1);

 \draw  (mat2-1-1) -- +(30:.7) node[anchor=west] {$p_{1,\ell(\mu)}$}
 (mat2-3-1) -- +(30:.7) node[anchor=west] {$p_{1,3}$}
 (mat2-4-1) -- +(30:.7) node[anchor=west] {$p_{1,2}$}
 (mat2-4-1) -- +(330:.7) node[anchor=west] {$p_{1,1}$};

  \end{tikzpicture}}

  Riemann-Hurwitz for the bottom-right vertex tells us that there must be
  \[
  2(a+b)-2-(a+b-2)-\left(a+b-\sum c_i-1+\sum(c_i-1)\right)=1+r
  \]
  additional ramification on that vertex.
  If $\ell(\mu')=0$ and there is a ramification profile of $(2^k)$ on the right side,
  then $1+r=k$ and $r=k-1$ with $a+b=2k$. Since there must be a ramification
  profile $(2^k)$ on the left, we have $c_i=2$ for all $i$ and the following
  picture:

    \scalebox{\BRANCHEDFACTOR}
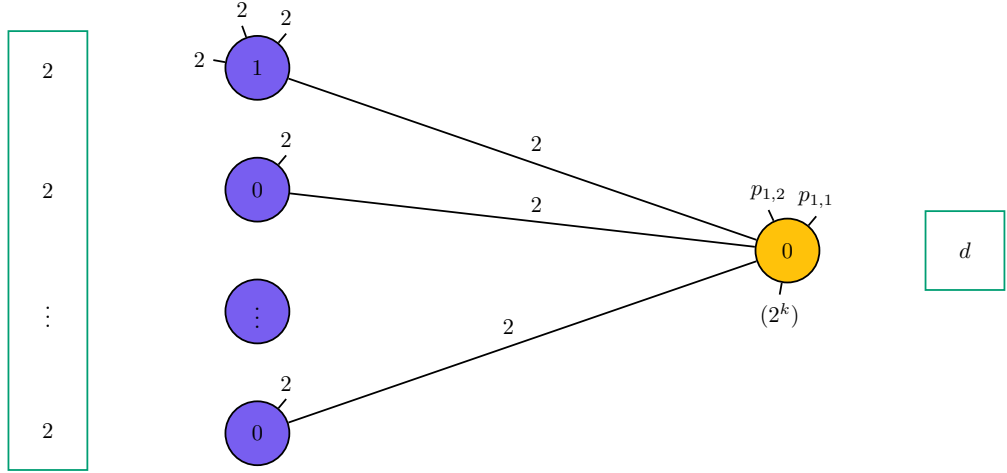
{
    \begin{tikzpicture}[thick,amat/.style={matrix of nodes,nodes in empty cells,
  row sep=2.5em,rounded corners,
  nodes={draw,solid,circle,minimum size=1.0cm}},
  dmat/.style={matrix of nodes,nodes in empty cells,row sep=2.5em,nodes={minimum size=1.0cm},draw=myred},
  fsnode/.style={fill=myblue},
  ssnode/.style={fill=mygreen}]

  \matrix[amat,nodes=fsnode] (mat1) {$1$\\
    $0$\\
    \vdots \\
  $0$\\};

  \matrix[dmat,left=2cm of mat1] (degrees1) {$2$\\
    $2$\\
    \vdots \\
  $2$\\};

  \matrix[amat,right=7cm of mat1,nodes=ssnode] (mat2) {$0$\\};
  
 \matrix[dmat,right=1.5cm of mat2] (degrees2) {$d$\\};

 \draw  (mat1-1-1) edge["$2$"] (mat2-1-1)
 (mat1-2-1) edge["$2$"] (mat2-1-1)
 (mat1-4-1) edge["$2$"] (mat2-1-1);

 \draw  (mat1-1-1) -- +(50:.7) node[anchor=south] {$2$}
 (mat1-1-1) -- +(110:.7) node[anchor=south] {$2$}
 (mat1-1-1) -- +(170:.7) node[anchor=east] {$2$}
 (mat1-2-1) -- +(50:.7) node[anchor=south] {$2$}
 (mat1-4-1) -- +(50:.7) node[anchor=south] {$2$};

 \draw  (mat2-1-1) -- +(50:.7) node[anchor=south] {$p_{1,1}$}
 (mat2-1-1) -- +(115:.7) node[anchor=south] {$p_{1,2}$}
 (mat2-1-1) -- +(260:.7) node[anchor=north] {$(2^k)$};

    \end{tikzpicture}}

    By Lemma~\ref{lem:twoscycles}, the Hurwitz number $H((2^k),(2^k),(a,b))$ on the right
    side vanishes unless $a=b$.
    If $a=b$, then the contribution of this picture to $W_{(2^k),(2^k),(2,1^{d-2})^{\ell(\mu)}}(a,a)$ is
    \begin{align*}
      &2\cdot k\cdot \binom{d-2}{2,\dots,2}^2\cdot W_{(2),(2),(2)}(2)\cdot \overline H((2),(2),(1,1),(1,1))^{k-1}\cdot \overline H((2^k),(2^k),(k,k))\cdot 2^{k-1} \\
      &=2k\cdot\frac{((d-2)!)^2}{2^{2k-2}}\cdot 6\cdot 2^{k-1}\cdot (k!)^2\cdot\frac 1k\cdot 2^{k-1} \\
      &=2\cdot 6\cdot ((d-2)!)^2\cdot\frac{2^{2k-2}}{2^{2k-2}}\cdot (k!)^2 \\
      &=12((d-2)!)^2\cdot (k!)^2
      \end{align*} 
    so the contribution to $N_{(2^k),(2^k)}(a,a)$ is
    $\frac{12((d-2)!)^2(k!)^2}{(k!)^2\cdot 2\cdot ((d-2)!)^2\cdot 2!}=3$ (corresponding to a choice of nontrivial 2-torsion point).
  
  Otherwise if $\ell(\mu')>0$ there cannot be a ramification profile of $(2^k)$
  on the right side, the additional ramification must be simple, and so $r=0$.
  Upon
  choosing which fiber of simple branching will appear on the right side and how it will appear,
  for which there are
  \[
  C=\ell(\mu)\cdot \binom{d-2}{a+b-2,\mu_3,\dots,\mu_{\ell(\mu)}}
  \]

  choices, there is a unique such map $\Gamma\to\Gamma'$, and $\Gamma$ has a unique $\hat A$-structure
  with $c_{\Gamma}=\prod\limits_{i=3}^{\ell(\mu)}\mu_i$, so the corresponding contribution to the right side of Theorem~\ref{thm:admissible} in this case is
  \begin{align*}
    &C\cdot W_{(2^k),(2^k),(2,1^{d-2})^{\ell(\mu)-1}}(a+b,\mu') \prod_{i=3}^{\ell(\mu)}\overline H((\mu_i),(\mu_i),(1^{\mu_i})) \cdot \overline H((a+b),(a,b),(2,1^{a+b-2})) c_{\Gamma} \\
    &=\frac{\ell(\mu)(d-2)!}{(a+b-2)!\mu_3!\cdots\mu_{\ell(\mu)}!}\cdot (k!)^2\cdot 2\cdot((d-2)!)^{\ell(\mu)-1} (\ell(\mu)-1)! N_{(2^k),(2^k)}(a+b,\mu')\cdot \prod_{i=3}^{\ell(\mu)}\mu_i!\cdot (a+b-2)! \tag*{(using Lemma~\ref{lem:Hurwitz2})} \\
    &=(\ell(\mu))!\cdot (k!)^2\cdot 2\cdot((d-2)!)^{\ell(\mu)}\cdot N_{(2^k),(2^k)}(a+b,\mu') \\
  \end{align*}

  and the left side is the $W$ invariant of interest, since $A$ has no nontrivial automorphisms. Thus,
  the contribution to $N_{(2^k),(2^k)}(a,b,\mu')$ is
  \begin{align*}
    \frac{(\ell(\mu))!\cdot (k!)^2\cdot 2\cdot((d-2)!)^{\ell(\mu)}\cdot N_{(2^k),(2^k)}(a+b,\mu')}{(k!)^2\cdot 2\cdot((d-2)!)^{\ell(\mu)}\cdot (\ell(\mu))!} =N_{(2^k),(2^k)}(a+b,\mu').
  \end{align*}

  Next, consider the case where $p_{1,1}$ and $p_{1,2}$ are not on the same component, in which
  case at least one vertex containing one of these two contains some additional ramification
  (necessarily simple).
  This then exhausts the branch points not arising from the genus $1$ vertex, so we have the following
  picture:

    \scalebox{\BRANCHEDFACTOR}
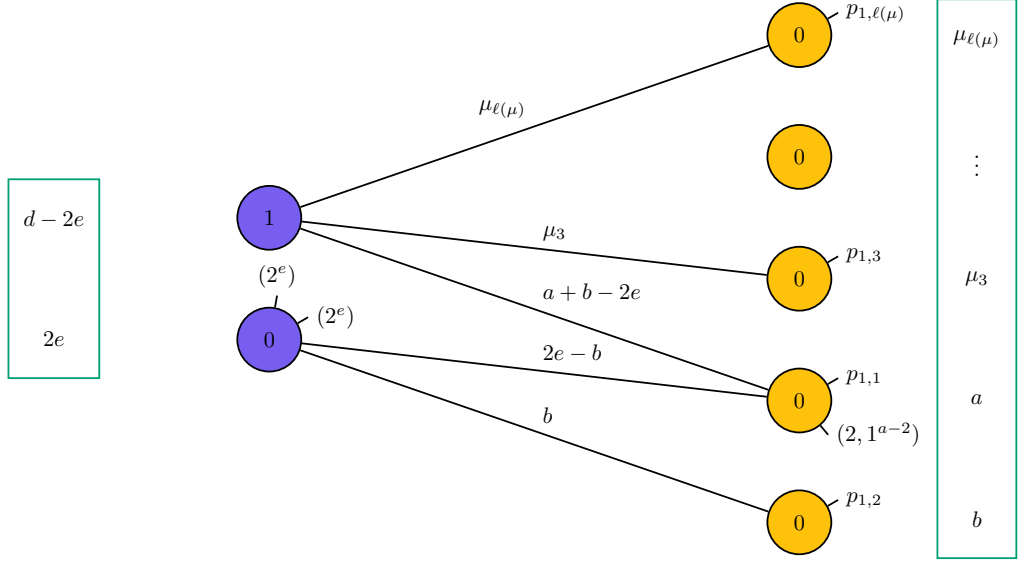
{
    \begin{tikzpicture}[thick,amat/.style={matrix of nodes,nodes in empty cells,
  row sep=2.5em,rounded corners,
  nodes={draw,solid,circle,minimum size=1.0cm}},
  dmat/.style={matrix of nodes,nodes in empty cells,row sep=2.5em,nodes={minimum size=1.0cm},draw=myred},
  fsnode/.style={fill=myblue},
  ssnode/.style={fill=mygreen}]

  \matrix[amat,nodes=fsnode] (mat1) {$1$\\
    $0$\\};

  \matrix[dmat,left=2cm of mat1] (degrees1) {$d-2e$\\
    $2e$\\};

 \matrix[amat,right=7cm of mat1,nodes=ssnode] (mat2) {$0$\\
   $0$ \\
   $0$ \\
   $0$\\
 $0$\\};
 \matrix[dmat,right=1.5cm of mat2] (degrees2) {$\mu_{\ell(\mu)}$\\
   \vdots\\
   $\mu_3$ \\
   $a$\\
 $b$\\};

 \draw  (mat1-1-1) edge["$\mu_{\ell(\mu)}$"] (mat2-1-1)
 (mat1-1-1) edge["$\mu_{3}$"] (mat2-3-1)
 (mat1-1-1) edge["$a+b-2e$"] (mat2-4-1)
 (mat1-2-1) edge["$2e-b$"] (mat2-4-1)
 (mat1-2-1) edge["$b$"] (mat2-5-1);

 \draw  (mat2-1-1) -- +(30:.7) node[anchor=west] {$p_{1,\ell(\mu)}$}
 (mat2-3-1) -- +(30:.7) node[anchor=west] {$p_{1,3}$}
 (mat2-4-1) -- +(30:.7) node[anchor=west] {$p_{1,1}$}
 (mat2-4-1) -- +(310:.7) node[anchor=west] {$(2,1^{a-2})$}
 (mat2-5-1) -- +(30:.7) node[anchor=west] {$p_{1,2}$};

 \draw  (mat1-2-1) -- +(30:.7) node[anchor=west] {$(2^e)$}
 (mat1-2-1) -- +(80:.7) node[anchor=south] {$(2^e)$};

    \end{tikzpicture}}

    By Lemma~\ref{lem:twoscycles}, $\overline H((2^e),(2^e),(2e-b,b))$ vanishes
    unless $b=e$ (which requires $a+b-2e=a-b>0\iff a>b$).
    Let $n_{a,b,d'}$ be the number of choices for the fibers in which to place
    the even ramification on the bottom left vertex. If $d-2b=a-b+d'$ (where
    $d'=|\mu'|$) is equal to $2$, then either $a-b=2$, $d'=0$, and $n_{a,b,d'}=\binom 32=3$
    since there are three ramified points on the genus $1$ vertex,
    or $a-b=d'=1$ and $n_{a,b,d'}=\binom 42=6$ since there are four ramified points on
    the genus $1$ vertex. If $d-2b>2$, then there are only two evenly ramified
    fibers on the genus $1$ vertex, so the placement of even ramification
    on the bottom left vertex is determined and $n_{a,b,d'}=1$.
    
    In any case the corresponding
    contribution to $W_{(2^k),(2^k),(2,1^{d-2})^{\ell(\mu)}}(a,b,\mu')$ is

    \begin{align*}
      &n_{a,b,d'}\ell(\mu)\binom{d-2}{a-2,b,\mu_3,\dots,\mu_{\ell(\mu)}}\binom{k}{b}^2\binom{d-2}{2b}^{\ell(\mu)-1}W_{(2^{k-b}),(2^{k-b}),(2,1^{d-2b-2})^{\ell(\mu)-1}}(a-b,\mu') \cdot \\
      &\overline H((2^b),(2^b),(b,b),(1^{2b})^{\ell(\mu)-1}) \cdot\prod_{i=3}^{\ell(\mu)}
      \overline H((\mu_i),(\mu_i),(1^{\mu_i}))\mu_i\cdot \overline H((a-b,b),(a),(2,1^{a-2}))\overline H((b),(b),(1^b))\cdot b^2 \\
      &=n_{a,b,d'}\ell(\mu)\frac{(d-2)!(k!)^2((d-2)!)^{\ell(\mu)-1}}{(a-2)!b!\mu_3!\cdots\mu_{\ell(\mu)}!(b!)^2((k-b)!)^2((2b)!(d-2b-2)!)^{\ell(\mu)-1}}N_{(2^{k-b}),(2^{k-b})}(a-b,\mu')\cdot \\
      &2((k-b)!)^2\cdot ((d-2b-2)!)^{\ell(\mu)-1}(\ell(\mu)-1)! \cdot\frac 1{2b}(b!)^2\cdot 2\cdot ((2b)!)^{\ell(\mu)-1}\cdot \mu_3!\cdots\mu_{\ell(\mu)}!(a-2)!b!\cdot b \\
      &=n_{a,b,d'}(\ell (\mu))!N_{(2^{k-b}),(2^{k-b})}(a-b,\mu')\cdot (k!)^2\cdot 2\cdot ((d-2)!)^{\ell(\mu)}
    \end{align*}
    so the contribution to $N_{(2^k),(2^k),(2,1^{d-2})^{\ell(\mu)}}(a,b,\mu')$ is $n_{a,b,d'}N_{(2^{k-b}),(2^{k-b})}(a-b,\mu')$.

    We can now turn to the situation where $p_{1,1},\dots,p_{1,\ell(\mu)}$ appear
    on the left side of $\Gamma$, i.e.\ where their images are on the same
    component as the image of the genus $1$ vertex. Since $p_{1,1}$ is not on the
    genus $1$ vertex, its path to the genus $1$ vertex must pass through
    a vertex on the right side, yielding at least two branch points.
    If any of $p_{1,3},\dots,p_{1,\ell(\mu)}$ were not on the genus $1$ vertex then
    this would be true of another vertex, implying that both fibers of
    ramification profile $(2^k)$ are on the right side and there is only
    simple ramification on the genus $1$ vertex (contradicting
    dimensional assumptions).

    If $p_{1,1}$ and $p_{1,2}$ are on different components, we have the following
    picture:

      \scalebox{\BRANCHEDFACTOR}
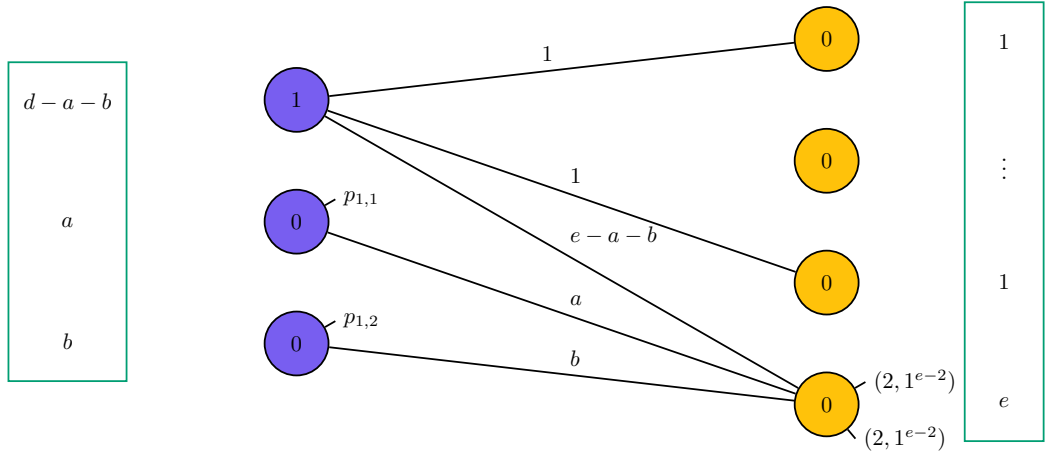
{
        \begin{tikzpicture}[thick,amat/.style={matrix of nodes,nodes in empty cells,
  row sep=2.5em,rounded corners,
  nodes={draw,solid,circle,minimum size=1.0cm}},
  dmat/.style={matrix of nodes,nodes in empty cells,row sep=2.5em,nodes={minimum size=1.0cm},draw=myred},
  fsnode/.style={fill=myblue},
  ssnode/.style={fill=mygreen}]

  \matrix[amat,nodes=fsnode] (mat1) {$1$\\
    $0$\\
  $0$\\};

  \matrix[dmat,left=2cm of mat1] (degrees1) {$d-a-b$\\
    $a$\\
  $b$\\};

 \matrix[amat,right=7cm of mat1,nodes=ssnode] (mat2) {$0$\\
   $0$ \\
   $0$ \\
   $0$\\};
 \matrix[dmat,right=1.5cm of mat2] (degrees2) {$1$\\
   \vdots\\
   $1$ \\
   $e$\\};

 \draw  (mat1-1-1) edge["$1$"] (mat2-1-1)
 (mat1-1-1) edge["$1$"] (mat2-3-1)
 (mat1-1-1) edge["$e-a-b$"] (mat2-4-1)
 (mat1-2-1) edge["$a$"] (mat2-4-1)
 (mat1-3-1) edge["$b$"] (mat2-4-1);

 \draw  
 (mat2-4-1) -- +(30:.7) node[anchor=west] {$(2,1^{e-2})$}
 (mat2-4-1) -- +(310:.7) node[anchor=west] {$(2,1^{e-2})$};

 \draw  (mat1-2-1) -- +(30:.7) node[anchor=west] {$p_{1,1}$}
 (mat1-3-1) -- +(30:.7) node[anchor=west] {$p_{1,2}$};

    \end{tikzpicture}}

        This is impossible by Riemann-Hurwitz for the bottom-right vertex, and so
        we are left with one remaining picture to consider where $p_{1,1}$ and
        $p_{1,2}$ appear on the same component (here $2e=a+b$):

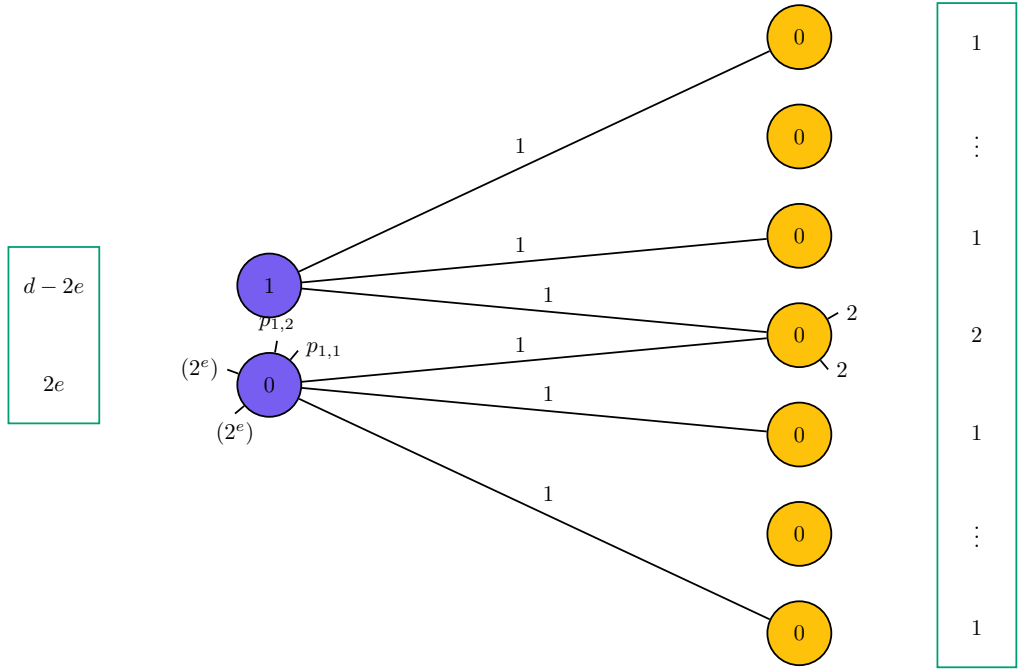
          \scalebox{\BRANCHEDFACTOR}{
                \begin{tikzpicture}[thick,amat/.style={matrix of nodes,nodes in empty cells,
  row sep=1.5em,rounded corners,
  nodes={draw,solid,circle,minimum size=1.0cm}},
  dmat/.style={matrix of nodes,nodes in empty cells,row sep=1.5em,nodes={minimum size=1.0cm},draw=myred},
  fsnode/.style={fill=myblue},
  ssnode/.style={fill=mygreen}]

  \matrix[amat,nodes=fsnode] (mat1) {$1$\\
    $0$\\};

  \matrix[dmat,left=2cm of mat1] (degrees1) {$d-2e$\\
    $2e$\\};

 \matrix[amat,right=7cm of mat1,nodes=ssnode] (mat2) {$0$\\
   $0$ \\
   $0$ \\
   $0$\\
   $0$\\
   $0$\\
 $0$\\};
 \matrix[dmat,right=1.5cm of mat2] (degrees2) {$1$\\
   \vdots\\
   $1$ \\
   $2$\\
   $1$\\
   \vdots\\
 $1$\\};

 \draw  (mat1-1-1) edge["$1$"] (mat2-1-1)
 (mat1-1-1) edge["$1$"] (mat2-3-1)
 (mat1-1-1) edge["$1$"] (mat2-4-1)
 (mat1-2-1) edge["$1$"] (mat2-4-1)
 (mat1-2-1) edge["$1$"] (mat2-5-1)
 (mat1-2-1) edge["$1$"] (mat2-7-1);

 \draw  
 (mat2-4-1) -- +(30:.7) node[anchor=west] {$2$}
 (mat2-4-1) -- +(310:.7) node[anchor=west] {$2$};

 \draw  (mat1-2-1) -- +(50:.7) node[anchor=west] {$p_{1,1}$}
 (mat1-2-1) -- +(80:.7) node[anchor=south] {$p_{1,2}$}
 (mat1-2-1) -- +(160:.7) node[anchor=east] {$(2^e)$}
 (mat1-2-1) -- +(220:.7) node[anchor=north] {$(2^e)$};

                \end{tikzpicture}}

                By Lemma~\ref{lem:twoscycles}, $\overline H((a,b),(2^e),(2^e))$ vanishes unless $a=b$.
                There are $\binom{\ell(\mu)}{2}$ choices for simple ramification profiles to place on the right
                side, $\binom{d-2}{2a-1}$ choices for which unramified points appear in the top right
                and bottom right corners of the diagram within the first chosen simple ramification profile,
                $(d-2)!$ choices for placement of unramified points in the second chosen
                right simple ramification profile, $\binom ka^2$ choices for placement of even ramification profiles
                on the left side, and $\binom{d-2}{2a}^{\ell(\mu)-2}$ choices for placement of unramified points in
                simple ramification profiles on the left side.
                As for choices of placement for even ramification:
                \begin{itemize}
                  \item
                If $\mu'=(2)$, then there are $3$ unmarked ramified points on
                the genus $1$ vertex, and two must be marked as being in
                even ramification fibers, yielding $\binom 32=3$ choices.
              \item If $\mu'=(1,1)$, then there are
                $4$ unmarked ramified points on the genus $1$ vertex,
                yielding $\binom 42=6$ choices.
              \item If $d'>2$, then there is a single choice, as the
                genus $1$ vertex unambiguously has two even ramification profiles.
                \end{itemize}

                We also have $c_{\Gamma}=2$, since there are two ways to contract
                edges of the graph to obtain a stabilization (one, but not
                both, of the edges from the degree $2$ vertex on the right
                will be contracted).
                There are also $d-2a-1$ unfixed points to be ordered
                on the genus $1$ vertex
                corresponding to all but one of its edges.
                Therefore the total
                contribution to $W_{(2^k),(2^k),(2,1^{d-2})^{\ell(\mu)}}(a,a,\mu')$ in this case is
                \begin{align*}
                  &\binom{\ell(\mu)}{2}\frac{((d-2)!)^2}{(2a-1)!(d-2a-1)!}\binom ka^2\binom{d-2}{2a}^{\ell(\mu)-2}W_{(2^{k-a}),(2^{k-a}),(2,1^{d-2a-2})^{\ell(\mu)-2}}(\mu')\cdot (d-2a-1)!\cdot  \\
                  &
                  \overline H((a,a),(2^a),(2^a),(1^{2a})^{\ell(\mu)-1}) \cdot
                  \overline H((2),(2),(1,1))\cdot 2 \\
                  &=\frac{(\ell(\mu))(\ell(\mu)-1)((d-2)!)^2}{2(2a-1)!}\frac{(k!)^2}{(a!(k-a)!)^2}\frac{((d-2)!)^{\ell(\mu)-2}}{((2a)!(d-2a-2)!)^{\ell(\mu)-2}}((k-a)!)^2\cdot 2 \cdot \\
                  &((d-2a-2)!)^{\ell(\mu)-2}\cdot (\ell(\mu)-2)!N_{(2^{k-a}),(2^{k-a})}(\mu')\cdot\frac{((2a)!)^{\ell(\mu)-1}((a!)^2)}{a}\cdot 2 \\
                  &=\frac{2(\ell(\mu))!((d-2)!)^{\ell(\mu)}(k!)^2N_{(2^{k-a}),(2^{k-a})}(\mu')
                  (2a)!}{(2a-1)!a} \\
                  &=2N_{(2^{k-a}),(2^{k-a})}(\mu')\cdot (k!)^2\cdot 2\cdot ((d-2)!)^{\ell(\mu)}\cdot (\ell(\mu))!
                \end{align*}
                multiplied by $3$ if $\mu'=(2)$ or by $6$ if $\mu'=(1,1)$.
                Therefore the corresponding contribution
                to $W_{(2^k),(2^k)}(a,a,\mu')$ is $2N_{(2^{k-a}),(2^{k-a})}(\mu')$
                if $d'>2$, $3\cdot 2\cdot 1=6$ if $\mu'=(2)$, and
                $6\cdot 2\cdot 1=12$ if $\mu'=(1,1)$.

\end{proof}

We can conclude with a proof of Theorem~\ref{thm:twos}:
\begin{proof}
  We can prove that $N_{(2^k),(2^k)}(\mu)=3\cdot 2^{\ell(\mu)-1}$ by induction on $\ell(\mu)$; the base case $\ell(\mu)=1$ was proven in Theorem~\ref{thm:combinatorial}. Assume
  the result for all $\nu$ with $\ell(\nu)< r$, and suppose $\ell(\mu)=r\geq 2$.
  Then we can write $\mu=(a,b,\mu')$ with $d'=|\mu'|$. If $a=b$, then we have
  four cases:
  \begin{enumerate}
  \item If $\mu'$ is empty: $N_{(2^k),(2^k)}(a,a)=N_{(2^k),(2^k)}(2a)+3=3+3=3\cdot 2^{r-1}$
  \item If $\mu'=(2)$: $N_{(2^k),(2^k)}(a,a,2)=N_{(2^k),(2^k)}(2a,2)+6=6+6=12=3\cdot 2^{r-1}$
  \item If $\mu'=(1,1)$: $N_{(2^k),(2^k)}(a,a,1,1)=N_{(2^k),(2^k)}(2a,1,1)+12=12+12=3\cdot 2^{r-1}$
  \item If $d'>2$:
      \begin{align*}
    N_{(2^k),(2^k)}(a,a,\mu') &= N_{(2^k),(2^k)}(2a,\mu')+2N_{(2^{k-a}),(2^{k-a})}(\mu') \\
    &=3\cdot 2^{r-2}+2(3\cdot 2^{r-3}) \\
    &=3\cdot 2^{r-1}
  \end{align*}

    \end{enumerate}
  Now if $a>b$ (without loss of generality), we have three cases:
  \begin{enumerate}
  \item If $a-b=2$ and $d'=0$:
    \[
    N_{(2^k),(2^k)}(a,a-2)=N_{(2^k),(2^k)}(2a-2)+3N_{(2),(2)}(2)=3+3\cdot 1=6=3\cdot 2^{r-1}
    \]
  \item If $a-b=1$ and $d'=1$:
    \[
    N_{(2^k),(2^k)}(a,a-1,1)=N_{(2^k),(2^k)}(2a-1,1)+6N_{(2),(2)}(1,1)=6+6\cdot 1=12=3\cdot 2^{r-1}
    \]
    \item If $d-2b>2$:
    \begin{align*}
      N_{(2^k),(2^k)}(a,b,\mu') &= N_{(2^k),(2^k)}(a+b,\mu')+2N_{(2^{k-b}),(2^{k-b})}(a-b,\mu') \\
      &= 3\cdot 2^{r-2}+3\cdot 2^{r-2} \\
      &= 3\cdot 2^{r-1}.
    \end{align*}
    \end{enumerate}
  \end{proof}

\section{Invariants $S_{\a,\b}$}
\label{section:Sab}

In this section, we consider a family of invariants $N_{g,\sigma,\nu}$ where
one ramification profile
$\mu$ is fully marked,
each other ramification profile $\sigma_i$ is of the form $(a,1,\dots,1)$, and
the only ``remembered'' points in these ramified fibers are some subset of the ones with
a nontrivial ramification index. We will establish general recursive formulas
(Theorems~\ref{thm:reduceb} and \ref{thm:reducea})
that allow us to compute this family of invariants, and will establish
specific closed-forms in special cases.

\subsection{Definitions and basic results}

Fix a degree $d\geq 1$. For any partition $\mu$ of $d$, and multisets $\a$ and $\b$ of
positive integers less than or equal to $d$, let
\[
S_{\a,\b}(\mu)=N_{1,\sigma,\nu}
\]
where
\[
\sigma=\{\mu\}\cup\{(x,1,\dots,1):x\in\a\cup\b\},\quad \nu=\{\mu\}\cup\{(x,1,\dots,1):x\in\b\}
\]
By Riemann-Hurwitz,
\[
d-\ell(\mu)+\sum_{x\in\a\cup\b}(x-1)=2d.
\]
We will use the convention that $x_1,\dots,x_{\ell(\mu)}$
form the fiber corresponding to $\mu$, $p_{i,1},\dots,p_{i,d-a_i+1}$ form
the fiber corresponding to the $i^{\text{th}}$ value $a_i\in\a$, and
$P_{i,1},\dots,P_{i,d-b_i+1}$ form the fiber corresponding to the
$i^{\text{th}}$ value $b_i\in\b$.

\begin{prop}
  Elements $1$ in $\b$ can be removed:
  $S_{\a,\b+\{1\}}(\mu)=S_{\a,\b}(\mu)$.
  We can therefore assume without loss of generality that $\b$ contains only integers greater than $1$.
\end{prop}
\begin{proof}
  The addition of a ramification profile $(1,\dots,1)$ does not affect
  our counting problem, since this is the generic form of a fiber;
  imposing that one marked point on the
  generic genus $1$
  marked curve is part of a generic fiber is no restriction at all.

  More formally,
  we have the following commutative diagram:

  \[\begin{tikzcd}
\Hb_{1,\sigma+\{(1,\dots,1)\}} & \overline{\mathcal M}_{1,\ell(\b)+1} \\
\Hb_{1,\sigma} & \overline{\mathcal M}_{1,\ell(\b)}
	\arrow[from=1-1, to=1-2, "\pi'"]
	\arrow[from=1-1, to=2-1]
	\arrow[from=1-2, to=2-2]
	\arrow[from=2-1, to=2-2, "\pi"]
  \end{tikzcd}\]
  This is not quite a Cartesian square, but the induced map
  \[\Hb_{1,\sigma+\{(1,\dots,1)\}}\to \Hb_{1,\sigma}\times_{\overline{\mathcal M}_{1,\ell(\b)}}\overline{\mathcal M}_{1,\ell(\b)+1}
  \]
  is $(d-1)!$-to-$1$ since a cover $f:E\to\P^1$ with an additional marked unramified fiber
  $y_1,\dots,y_d$
  is equivalent to the cover with the marked point $y_1$ and an orientation
  of the points it determines in its fiber $f^{-1}(f(y_1))$, $\{y_2,\dots,y_d\}$.
  Therefore
  \[
  W'=\text{deg}(\pi')=(d-1)!\cdot\text{deg}(\pi)=(d-1)!\cdot W
  \]
  where $W$ and $W'$ are the labeled versions of $S_{\a,\b}(\mu)$ and
  $S_{\a,\b+\{1\}}(\mu)$ respectively, and incorporating the normalizing factors
  yields the result.
  \end{proof}

The special case $\b=\emptyset$ is of most interest to us, as it can be rewritten
\[
S_{\a,\emptyset}(\mu)=N_{(a_1,1,\dots,1),\dots,(a_{\ell(\a)},1,\dots,1)}(\mu)
\]
However even if we wish to focus solely on this special case, our
recursive formulas require us to consider the more general notion (see
Theorem~\ref{thm:reducea}).

\begin{lem}
  \label{lem:Sabdim}
  If $S_{\mathfrak a,\mathfrak b}(\mu)$ is well-defined, $|\mathfrak a|=\ell(\mu)+2$.
\end{lem}
\begin{proof}
  By Lemma~\ref{lem:dim}, $3+m=n$ where $m$ is the number of points marked under $\nu$ and
  $n$ is the number of branch points. In this case $n=|\mathfrak a|+|\mathfrak b|+1$
  and $m=\ell(\mu)+|\mathfrak b|$, so
  \[
  3+\ell(\mu)+|\mathfrak b|=|\mathfrak a|+|\mathfrak b|+1\implies |\mathfrak a|=\ell(\mu)+2.
  \]
\end{proof}

\begin{eg}
  Let $d=4$, let $\mu=(4)$, let $\a=\{3,2,2\}$ and let $\b=\{2\}$. Then
  \[
  S_{\a,\b}(\mu)=N_{1,((4),(3,1),(2,1,1),(2,1,1)),((4),(),(),(2))}
  \]
  This counts genus $1$ covers of $\P^1$ as illustrated in Figure~\ref{fig:Sab} (with pink, bolded points fixed).

  \begin{figure}[h]
    \caption{Example of a cover counted by $S_{\{3,2,2\},\{2\}}(4)$}
  \centering
\begin{tikzpicture}
  \draw[domain=-7:7,samples=50,color=gray!50,thick] plot (\x, \x^3/64 - \x/2);
  \foreach \x/\p/\c/\d in {-6/$\mathbf{x_1}$/mypink/4}
  {
    \filldraw[\c] (\x,\x^3/64 - \x/2) circle (0.6mm) node[above] {\p};
    \draw[->] (\x,\x^3/64 - \x/2 - .4) to["\d"] (-6, -3.5);
  }

    \foreach \x/\p/\c/\d in {-4.8/$p_{1,1}$/black/3, -3.5/$p_{1,2}$/black/1}
  {
    \filldraw[\c] (\x,\x^3/64 - \x/2) circle (0.6mm) node[above] {\p};
    \draw[->] (\x,\x^3/64 - \x/2 - .4) to["\d"] (-2.8+.3*\x, -3.5);
  }
  
  \foreach \x/\p/\c/\d in {-1.8/$p_{2,1}$/black/2, -.9/$p_{2,2}$/black/1, .2/$p_{2,3}$/black/1}
  {
    \filldraw[\c] (\x,\x^3/64 - \x/2) circle (0.6mm) node[above] {\p};
    \draw[->] (\x,\x^3/64 - \x/2 - .4) to["\d"] (-1+.3*\x, -3.5);
  }

    \foreach \x/\p/\c/\d in {1/$p_{3,1}$/black/2, 2/$p_{3,2}$/black/1, 3/$p_{3,3}$/black/1}
  {
    \filldraw[\c] (\x,\x^3/64 - \x/2) circle (0.6mm) node[above] {\p};
    \draw[->] (\x,\x^3/64 - \x/2 - .4) to["\d"] (1+.3*\x, -3.5);
  }
  
  \foreach \x/\p/\c/\d in {4/$\mathbf{P_{1,1}}$/mypink/2, 4.8/$P_{1,2}$/black/1, 5.7/$P_{1,3}$/black/1}
  {
    \filldraw[\c] (\x,\x^3/64 - \x/2) circle (0.6mm) node[above] {\p};
    \draw[->] (\x,\x^3/64 - \x/2 - .4) to["\d"] (3+.3*\x, -3.5);
    }

    \draw[domain=-7:7,samples=50,color=gray!50,thick] plot (\x, -4);
  \foreach \x/\q in {-6/$q_1$, -4/$q_2$, -1.4/$q_3$, 1.4/$q_4$, 4.3/$q_5$}
    \filldraw[black] (\x,-4) circle (0.6mm) node[below] {\q};
\end{tikzpicture}

\label{fig:Sab}
  \end{figure}
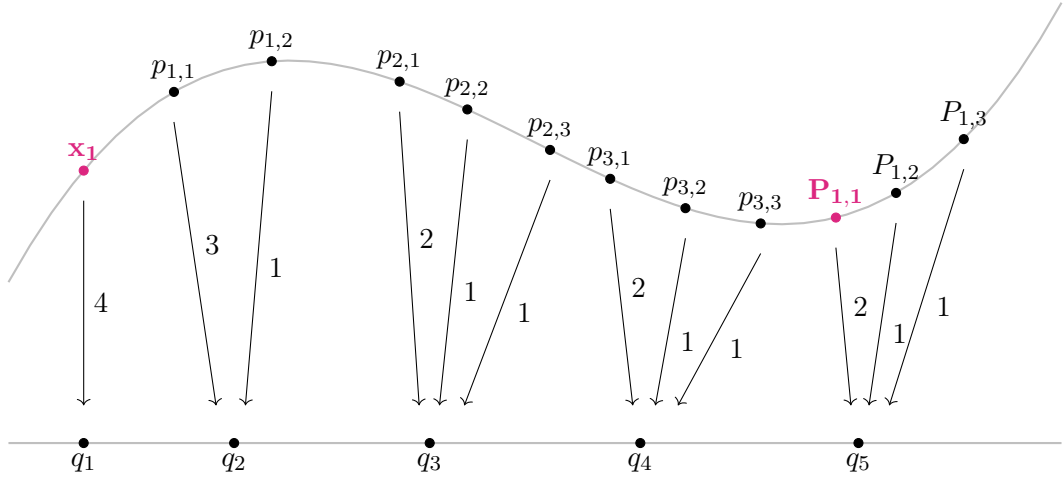
\end{eg}

\begin{eg}
  \label{eg:Sabns}
  We can specialize to the case where $\a$ is of the simple form
  $\a=\{n,\dots,n\}\cup\{2,\dots,2\}$, where there are $k$ appearances
  of $n$ and $\ell$ appearances of $2$, and $\b=\{t\}$ (the $\b=\emptyset$
  case, which is of particular interest to us, corresponding to $t=1$).
  By Riemann-Hurwitz and Lemma~\ref{lem:Sabdim}, for
  $S_{\{n\}^k\cup\{2\}^{\ell},\{t\}}(\mu)$ to be well-defined,
  \begin{align*}
    d-(\ell(\a)-2)+\sum_{x\in\a}(x-1)+(t-1)=2d &\iff d-(k+\ell-2)+k(n-1)+\ell+t-1=2d \\
    &\iff k(n-2)+1+t=d
  \end{align*}
  and $k+\ell=\ell(\mu)+2$.
  For example in the case $n=3$ and $t=1$, we have $d=k+2$; for any fixed $d$,
  \[S_{\{3\}^{d-2}\cup\{2\}^{\ell(\mu)+4-d},\emptyset}(\mu)\] counts genus $1$ covers
  with a fixed ramification profile $\mu$ and as many other
  points of ramification index $3$ as possible. It is well-defined
  as long as $d-\ell(\mu)\leq 4$.
  In Example~\ref{eg:fourtorsion}, we saw that
  \[
  S_{\{4\}\cup\{2,2\},\emptyset}(4)=N_{(4)}(4)=15.
  \]
  For $n>3$, there are additional
  restrictions on the degree $d$.
  \end{eg}

We will compute the invariants $S_{\a,\b}(\mu)$ as follows. First,
we use our admissible covers technique
and Theorem~\ref{thm:admissible} to establish a recursion (Theorem~\ref{thm:reduceb}) that allows us to remove elements of
$\b$, and we then apply the same
technique to establish another recursion (Theorem~\ref{thm:reducea}) that allows us to remove elements of $\a$; the combination
of these recursions allows us to reduce to the case where $\mu=(d)$ and $\b$ is empty. Finally,
we establish the base case (Theorem~\ref{thm:Sabbase}) using our previously demonstrated genus reduction technique.

While these results in principle provide an algorithm for computing any invariant
$S_{\a,\b}(\mu)$, the recursions involved are too messy to expect a general closed
form solution for these invariants,
even for $\mu=(1^d)$. However, we demonstrate
how they can be used to find a closed form for a special family
of invariants in Theorem~\ref{thm:deg4}.

\subsection{Reduction of $\b$}

\begin{thm}
  \label{thm:reduceb}
  For $n>0$, if $y>0$, then
  \begin{align*}
    S_{\mathfrak a,\mathfrak b+\{y\}}(\mu, n) &=S_{\mathfrak a,\mathfrak b}(\mu, n-y+1) +\sum_{z\in\mathfrak a}\min(z-1,y-1,n,y+z-n-2)S_{\mathfrak a-\{z\},\mathfrak b+\{y+z-n-2\}}(\mu)
  \end{align*}
  This sum is over $\a$ as a set rather than a multiset, i.e.\ there is
  a single term for each unique value in $\a$.
\end{thm}

\begin{proof}
  Let $k=\ell(\mu)$. We can apply Theorem~\ref{thm:admissible}
  to calculate $W_{1,\sigma,\nu}$ and divide by the normalizing factor
  \[ 
  \prod_{x\in\a+\b+\{y\}}(d-x)!\cdot |\Aut(\a)|\cdot |\Aut(\b+\{y\})|,
  \]
  but for brevity we will instead perform calculations on level of
  $N_{1,\sigma,\nu}$ to avoid writing combinatorial terms that cancel.
  Let $A$ be the following graph,
  with a genus $1$ component containing $x_1,\dots,x_k$ and all but one
  of the other fixed points:

    \scalebox{\BRANCHEDFACTOR}
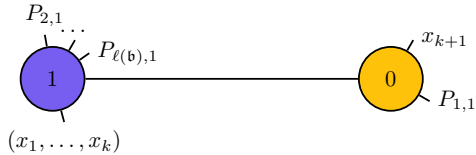
{
  \begin{tikzpicture}[thick,amat/.style={matrix of nodes,nodes in empty cells,
  row sep=2.5em,rounded corners,
  nodes={draw,solid,circle,minimum size=1.0cm}},
  dmat/.style={matrix of nodes,nodes in empty cells,row sep=2.5em,nodes={minimum size=1.0cm},draw=myred},
  fsnode/.style={fill=myblue},
  ssnode/.style={fill=mygreen}]

  \matrix[amat,nodes=fsnode] (mat1) {$1$\\};

 \matrix[amat,right=4cm of mat1,nodes=ssnode] (mat2) {$0$\\};

 \draw  (mat1-1-1) edge (mat2-1-1);

 \draw (mat1-1-1) -- +(35:.7) node[anchor=west] {$P_{\ell(\b),1}$}
 (mat1-1-1) -- +(60:.7) node[anchor=south] {$\dots$}
 (mat1-1-1) -- +(100:.7) node[anchor=south] {$P_{2,1}$}
 (mat1-1-1) -- +(285:.7) node[anchor=north] {$(x_1,\dots,x_k)$};

 \draw  (mat2-1-1) -- +(60:.7) node[anchor=west] {$x_{k+1}$}
 (mat2-1-1) -- +(330:.7) node[anchor=west] {$P_{1,1}$};

                \end{tikzpicture}}

                  Note that $|\Aut(A)|=1$.
We now aim to count admissible covers of graphs $\Gamma\to\Gamma'$ where $\Gamma$ has a $\hat A$-structure and $\hat A$ specializes to $A$.
Recall by Lemma~\ref{lem:branchpoints} that there are two branch points not
  arising from the genus $1$ component.
  
  A cover of graphs with $x_1,\dots,x_k$ on the right side looks like:

    \scalebox{\BRANCHEDFACTOR}
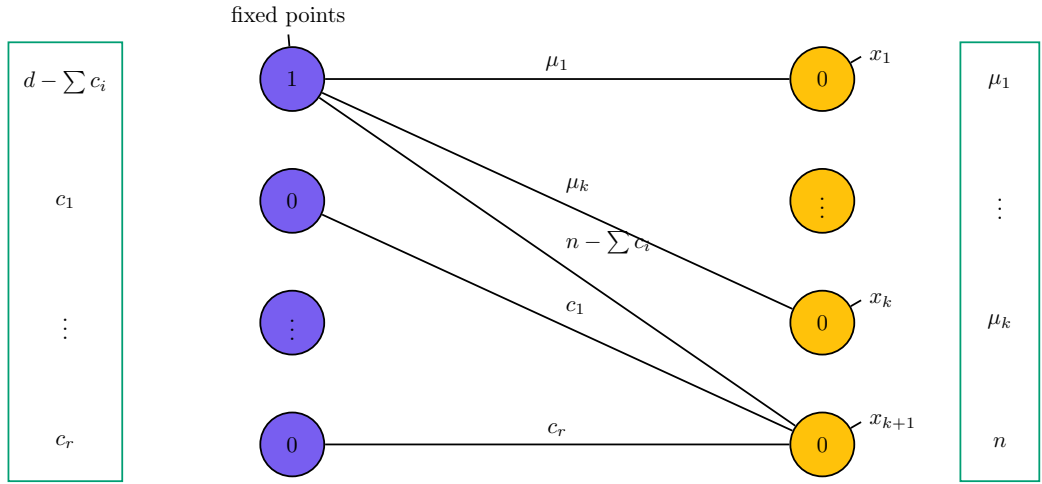
{
\begin{tikzpicture}[thick,amat/.style={matrix of nodes,nodes in empty cells,
  row sep=2.5em,rounded corners,
  nodes={draw,solid,circle,minimum size=1.0cm}},
  dmat/.style={matrix of nodes,nodes in empty cells,row sep=2.5em,nodes={minimum size=1.0cm},draw=myred},
  fsnode/.style={fill=myblue},
  ssnode/.style={fill=mygreen}]

  \matrix[amat,nodes=fsnode] (mat1) {$1$\\
    $0$\\
    \vdots \\
  $0$\\};

  \matrix[dmat,left=2cm of mat1] (degrees1) {$d-\sum c_i$\\
    $c_1$\\
    \vdots \\
  $c_r$\\};

 \matrix[amat,right=7cm of mat1,nodes=ssnode] (mat2) {$0$\\
   \vdots\\
   $0$ \\
   $0$\\};
 \matrix[dmat,right=1.5cm of mat2] (degrees2) {$\mu_1$\\
   \vdots\\
   $\mu_k$ \\
 $n$\\};

 \draw  (mat1-1-1) edge["$\mu_1$"] (mat2-1-1)
 (mat1-1-1) edge["$\mu_k$"] (mat2-3-1)
 (mat1-1-1) edge["$n-\sum c_i$"] (mat2-4-1)
 (mat1-2-1) edge["$c_1$"] (mat2-4-1)
 (mat1-4-1) edge["$c_r$"] (mat2-4-1);

 \draw (mat1-1-1) -- +(95:.7) node[anchor=south] {fixed points};

 \draw  (mat2-1-1) -- +(30:.7) node[anchor=west] {$x_1$}
 (mat2-3-1) -- +(30:.7) node[anchor=west] {$x_k$}
 (mat2-4-1) -- +(30:.7) node[anchor=west] {$x_{k+1}$};

\end{tikzpicture}}

There are two branch points outside the genus $1$ component, which must be the images of
$x_1,\dots,x_k,x_{k+1}$ and of $P_{1,1}$. This means there is no extra ramification on any genus $0$ component, which is why each top right bubble
connects only to the genus $1$ component. The bottom right bubble must have some ramification beyond what is pictured here, which is necessarily
$P_{1,1}$. So our diagram looks like:

  \scalebox{\BRANCHEDFACTOR}
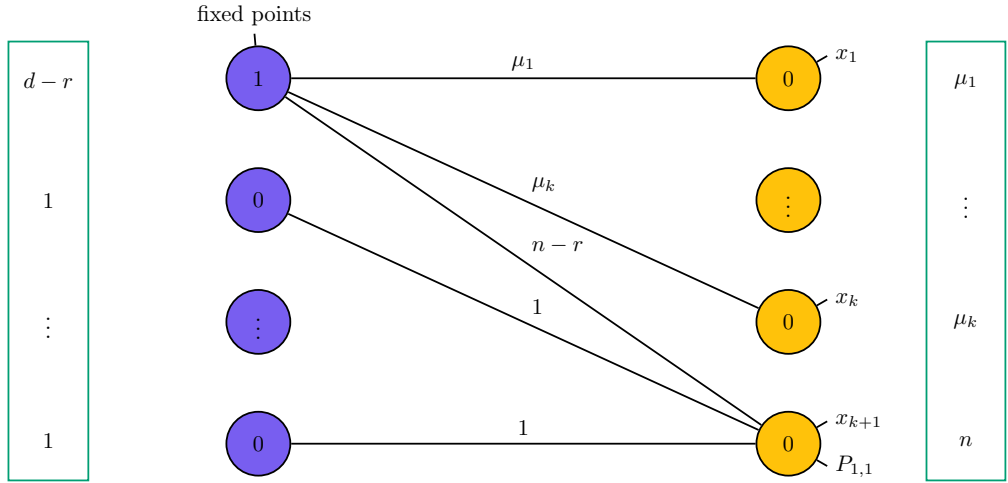
{
  \begin{tikzpicture}[thick,amat/.style={matrix of nodes,nodes in empty cells,
  row sep=2.5em,rounded corners,
  nodes={draw,solid,circle,minimum size=1.0cm}},
  dmat/.style={matrix of nodes,nodes in empty cells,row sep=2.5em,nodes={minimum size=1.0cm},draw=myred},
  fsnode/.style={fill=myblue},
  ssnode/.style={fill=mygreen}]

  \matrix[amat,nodes=fsnode] (mat1) {$1$\\
    $0$\\
    \vdots \\
  $0$\\};

  \matrix[dmat,left=2cm of mat1] (degrees1) {$d-r$\\
    $1$\\
    \vdots \\
  $1$\\};

 \matrix[amat,right=7cm of mat1,nodes=ssnode] (mat2) {$0$\\
   \vdots\\
   $0$ \\
   $0$\\};
 \matrix[dmat,right=1.5cm of mat2] (degrees2) {$\mu_1$\\
   \vdots\\
   $\mu_k$ \\
 $n$\\};

 \draw  (mat1-1-1) edge["$\mu_1$"] (mat2-1-1)
 (mat1-1-1) edge["$\mu_k$"] (mat2-3-1)
 (mat1-1-1) edge["$n-r$"] (mat2-4-1)
 (mat1-2-1) edge["$1$"] (mat2-4-1)
 (mat1-4-1) edge["$1$"] (mat2-4-1);

 \draw (mat1-1-1) -- +(95:.7) node[anchor=south] {fixed points};

 \draw  (mat2-1-1) -- +(30:.7) node[anchor=west] {$x_1$}
 (mat2-3-1) -- +(30:.7) node[anchor=west] {$x_k$}
 (mat2-4-1) -- +(30:.7) node[anchor=west] {$x_{k+1}$}
 (mat2-4-1) -- +(330:.7) node[anchor=west] {$P_{1,1}$};

  \end{tikzpicture}}

  As there is no extra ramification outside the genus $1$ component,
  the dimension reduction of $r$ should match the codimension reduction
  corresponding to the removal of $P_{1,1}$, which is $y-1$. So $r=y-1$ and we
  have the first term in the claim, since $H((n),(y,1,\dots),(n-y+1,1,\dots))=1$.

  Now consider the case where $x_1,\dots,x_k,x_{k+1}$ are on the left side. In this
  case, the branch points not coming from the genus $1$ component will be
  the image of $P_{1,1}$ and the image of another point $p_{i,1}$ (with ramification index $z\in\mathfrak a$). Consider the possibility that
  $x_1,\dots,x_k$ are not all on the genus $1$ component:

    \scalebox{\BRANCHEDFACTOR}
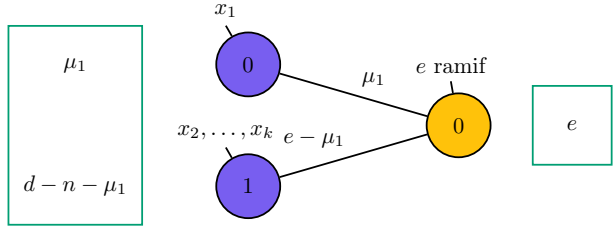
{
  \begin{tikzpicture}[thick,amat/.style={matrix of nodes,nodes in empty cells,
  row sep=2.5em,rounded corners,
  nodes={draw,solid,circle,minimum size=1cm}},
  dmat/.style={matrix of nodes,nodes in empty cells,row sep=2.5em,nodes={minimum size=1cm},draw=myred},
  fsnode/.style={fill=myblue},
  ssnode/.style={fill=mygreen}]

    \matrix[amat,nodes=fsnode] (mat1) {$0$\\
      $1$\\};

    \matrix[amat,right=2cm of mat1,nodes=ssnode] (mat2) {$0$\\};

      \matrix[dmat,left=1cm of mat1] (degrees1) {$\mu_1$\\
    $d-n-\mu_1$\\};

            \matrix[dmat,right=0.5cm of mat2] (degrees2) {$e$\\};
  
 \draw  (mat1-1-1) edge["$\mu_1$"] (mat2-1-1)
 (mat1-2-1) edge["$e-\mu_1$"] (mat2-1-1);

 \draw (mat1-1-1) -- +(120:.7) node[anchor=south] {$x_1$}
 (mat1-2-1) -- +(120:.7) node[anchor=south] {$x_2,\dots,x_k$};

 \draw (mat2-1-1) -- +(100:.7) node[anchor=south] {$e$ ramif};

  \end{tikzpicture}}

  This would require the genus $0$ on the right side to have multiple branch points contributing
  $e$ ramification,
  but $P_{1,1}$ cannot live on this component (because of the stabilization condition), a contradiction.
  Therefore, $x_1,\dots,x_k$ all live on the genus $1$ component.

  The points $P_{1,1}$ and $p_{i,1}$ cannot lie on separate components, because the following diagram cannot be completed:

    \scalebox{\BRANCHEDFACTOR}
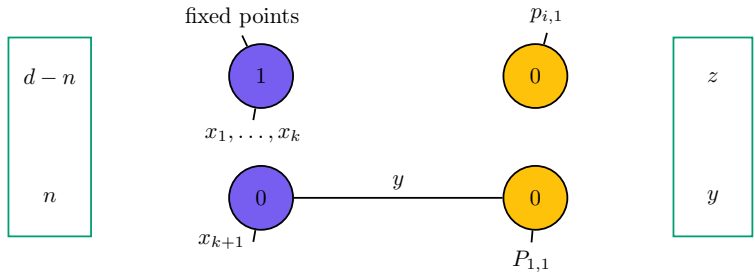
{
        \begin{tikzpicture}[thick,amat/.style={matrix of nodes,nodes in empty cells,
  row sep=2.5em,rounded corners,
  nodes={draw,solid,circle,minimum size=1.0cm}},
  dmat/.style={matrix of nodes,nodes in empty cells,row sep=2.5em,nodes={minimum size=1.0cm},draw=myred},
  fsnode/.style={fill=myblue},
  ssnode/.style={fill=mygreen}]

  \matrix[amat,nodes=fsnode] (mat1) {$1$\\
    $0$\\};

  \matrix[dmat,left=2cm of mat1] (degrees1) {$d-n$\\
    $n$\\};

 \matrix[amat,right=3cm of mat1,nodes=ssnode] (mat2) {$0$\\
   $0$\\};
 \matrix[dmat,right=1.5cm of mat2] (degrees2) {$z$\\
   $y$\\};

 \draw  (mat1-2-1) edge["$y$"] (mat2-2-1);

 \draw (mat1-1-1) -- +(115:.7) node[anchor=south] {fixed points}
 (mat1-1-1) -- +(260:.7) node[anchor=north] {$x_1,\dots,x_k$}
 (mat1-2-1) -- +(260:.7) node[anchor=east] {$x_{k+1}$};

 \draw (mat2-1-1) -- +(75:.7) node[anchor=south] {$p_{i,1}$}
 (mat2-2-1) -- +(265:.7) node[anchor=north] {$P_{1,1}$};

 \end{tikzpicture}}

  Therefore $P_{1,1}$ and $p_{i,1}$ lie on the same genus $0$ component, all other genus $0$ components on the right side have degree $1$,
  and we end up with:

    \scalebox{\BRANCHEDFACTOR}
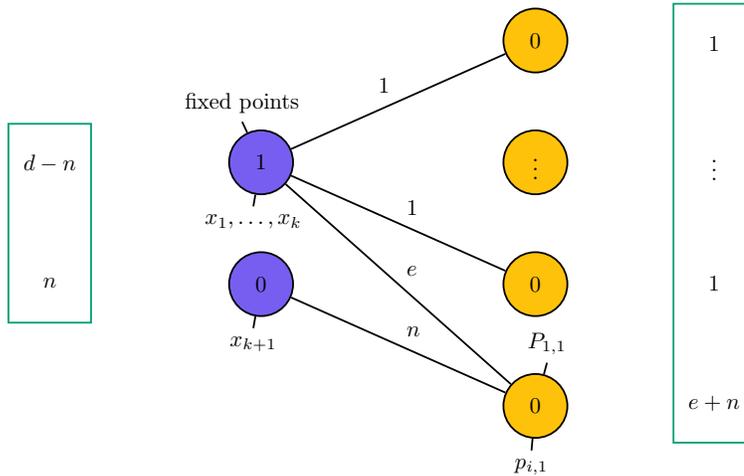
{
        \begin{tikzpicture}[thick,amat/.style={matrix of nodes,nodes in empty cells,
  row sep=2.5em,rounded corners,
  nodes={draw,solid,circle,minimum size=1.0cm}},
  dmat/.style={matrix of nodes,nodes in empty cells,row sep=2.5em,nodes={minimum size=1.0cm},draw=myred},
  fsnode/.style={fill=myblue},
  ssnode/.style={fill=mygreen}]

  \matrix[amat,nodes=fsnode] (mat1) {$1$\\
    $0$\\};

  \matrix[dmat,left=2cm of mat1] (degrees1) {$d-n$\\
    $n$\\};

 \matrix[amat,right=3cm of mat1,nodes=ssnode] (mat2) {$0$\\
   \vdots\\
   $0$ \\
   $0$\\};
 \matrix[dmat,right=1.5cm of mat2] (degrees2) {$1$\\
   \vdots\\
   $1$ \\
   $e+n$\\};

 \draw  (mat1-1-1) edge["$1$"] (mat2-1-1)
 (mat1-1-1) edge["$1$"] (mat2-3-1)
 (mat1-1-1) edge["$e$"] (mat2-4-1)
 (mat1-2-1) edge["$n$"] (mat2-4-1);

 \draw (mat1-1-1) -- +(115:.7) node[anchor=south] {fixed points}
 (mat1-1-1) -- +(260:.7) node[anchor=north] {$x_1,\dots,x_k$}
 (mat1-2-1) -- +(260:.7) node[anchor=north] {$x_{k+1}$};

 \draw (mat2-4-1) -- +(75:.7) node[anchor=south] {$P_{1,1}$}
 (mat2-4-1) -- +(265:.7) node[anchor=north] {$p_{i,1}$};

        \end{tikzpicture}}

        By Riemann-Hurwitz for the bottom-right component, 
        \[
        (e+n-2)+(y-1)+(z-1)=2(e+n)-2\implies y+z-2=e+n\implies e=y+z-n-2.
        \]
        By our calculation of $H((y+z-n-2,n),(z,1,\dots),(y,1,\dots))$ (and noting that
        we should multiply by $2$ if $y+z-n-2=n$, since the corresponding edges should be
        distinguished), we end up with the term in the claim.
        
\end{proof}

\subsection{Reduction of $\a$}
\begin{thm}
  \label{thm:reducea}
  For $n>0$,
  \begin{align*}
    S_{\a,\b}(\mu, n, m) &= \sum_{\substack{z\in\mathfrak a,\ z\leq n+m \\ n+m-z+1>0}}\min(n,m,z-1,n+m-z+1)S_{\a-\{z\},\b}(\mu,n+m-z+2) \\
    &+\sum_{\substack{z_1,z_2\in\a-\{2\} \\ z_1+z_2-n-m-3>0}}X_{z_1,z_2,n,m}S_{\a-\{z_1,z_2\},\b+\{z_1+z_2-n-m-3\}}(\mu)
  \end{align*}
  where
  \[
  X_{z_1,z_2,n,m}=H((z_1+z_2-n-m-3,n,m),(z_1,1,\dots),(z_2,1,\dots))\cdot \frac{|\Aut(n,m,z_1+z_2-n-m-3)|}{|\Aut(z_1,z_2)|}
  \]
  These sums are over $\a$ as a set rather than a multiset. 
\end{thm}

\begin{proof}
    We consider diagrams which stabilize to a genus $1$ curve containing $x_1,\dots,x_k$ and all 
  of the other fixed points, attached to a genus $0$ tail containing $q$ (with ramification index $n$) and $r$
  (with ramification index $m$).
  
  If $x_1,\dots,x_k,q,r$ are on the right side, then there is one more branch point arising as the image of some $s$ (with
  ramification index $z$). Suppose $q$ and $r$ are on different components:

    \scalebox{\BRANCHEDFACTOR}
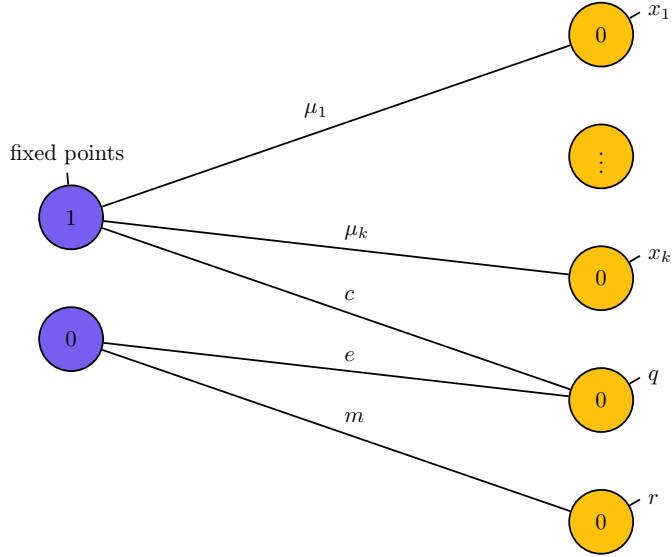
{
  \begin{tikzpicture}[thick,amat/.style={matrix of nodes,nodes in empty cells,
  row sep=2.5em,rounded corners,
  nodes={draw,solid,circle,minimum size=1.0cm}},
  dmat/.style={matrix of nodes,nodes in empty cells,row sep=2.5em,nodes={minimum size=1.0cm},draw=myred},
  fsnode/.style={fill=myblue},
  ssnode/.style={fill=mygreen}]

  \matrix[amat,nodes=fsnode] (mat1) {$1$\\
    $0$\\};

 \matrix[amat,right=7cm of mat1,nodes=ssnode] (mat2) {$0$\\
   \vdots\\
   $0$ \\
   $0$\\
 $0$\\};

 \draw  (mat1-1-1) edge["$\mu_1$"] (mat2-1-1)
 (mat1-1-1) edge["$\mu_k$"] (mat2-3-1)
 (mat1-1-1) edge["$c$"] (mat2-4-1)
 (mat1-2-1) edge["$e$"] (mat2-4-1)
 (mat1-2-1) edge["$m$"] (mat2-5-1);

 \draw (mat1-1-1) -- +(95:.7) node[anchor=south] {fixed points};

 \draw  (mat2-1-1) -- +(30:.7) node[anchor=west] {$x_1$}
 (mat2-3-1) -- +(30:.7) node[anchor=west] {$x_k$}
 (mat2-4-1) -- +(30:.7) node[anchor=west] {$q$}
 (mat2-5-1) -- +(30:.7) node[anchor=west] {$r$};

  \end{tikzpicture}}

  (Without loss of generality, we suppose the component containing $q$ is the one that connects to the genus $1$.)
  There is extra ramification on both the component containing $q$ and the lower left component, a contradiction. Therefore
  $q$ and $r$ lie on the same component, and our picture must look like:

    \scalebox{\BRANCHEDFACTOR}
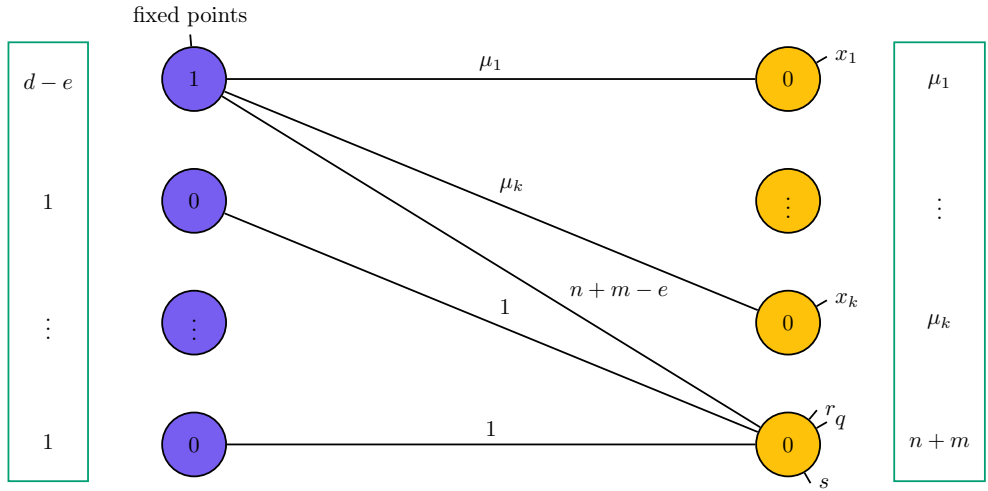
{
    \begin{tikzpicture}[thick,amat/.style={matrix of nodes,nodes in empty cells,
  row sep=2.5em,rounded corners,
  nodes={draw,solid,circle,minimum size=1.0cm}},
  dmat/.style={matrix of nodes,nodes in empty cells,row sep=2.5em,nodes={minimum size=1.0cm},draw=myred},
  fsnode/.style={fill=myblue},
  ssnode/.style={fill=mygreen}]

  \matrix[amat,nodes=fsnode] (mat1) {$1$\\
    $0$\\
    \vdots\\
    $0$\\};

    \matrix[dmat,left=1cm of mat1] (degrees1) {$d-e$\\
    $1$\\
    \vdots \\
  $1$\\};

 \matrix[amat,right=8cm of mat1,nodes=ssnode] (mat2) {$0$\\
   \vdots\\
   $0$ \\
   $0$\\};

     \matrix[dmat,right=1cm of mat2] (degrees2) {$\mu_1$\\
    \vdots \\
    $\mu_k$\\
     $n+m$\\};

 \draw  (mat1-1-1) edge["$\mu_1$"] (mat2-1-1)
 (mat1-1-1) edge["$\mu_k$"] (mat2-3-1)
 (mat1-1-1) edge["$n+m-e$" shift={(1.1,-0.7)}] (mat2-4-1)
 (mat1-2-1) edge["$1$"] (mat2-4-1)
 (mat1-4-1) edge["$1$"] (mat2-4-1);

 \draw (mat1-1-1) -- +(95:.7) node[anchor=south] {fixed points};

 \draw  (mat2-1-1) -- +(30:.7) node[anchor=west] {$x_1$}
 (mat2-3-1) -- +(30:.7) node[anchor=west] {$x_k$}
 (mat2-4-1) -- +(30:.7) node[anchor=west] {$q$}
 (mat2-4-1) -- +(50:.7) node[anchor=west] {$r$}
 (mat2-4-1) -- +(300:.7) node[anchor=west] {$s$};

    \end{tikzpicture}}

    By Riemann-Hurwitz for the lower right component,
    \[
    (n+m-e-1)+(n+m-2)+(z-1)=2(n+m)-2\implies e=z-2
    \]
    yielding the first term in the claim using Lemma~\ref{lem:Hurwitz2} to simplify
    $H((n,m),(z,1,\dots),(n+m-z+2,1,\dots))$.

    If $x_1,\dots,x_k,q,r$ are on the left side, then there are two branch points arising as the images of $s_1$ (with
    ramification index $z_1$) and $s_2$ (with ramification index $z_2$). By logic similar to the last proof, the
    $x_1,\dots,x_k$ all lie on the genus $1$ component.

    Consider the possibility that $q$ and $r$ lie on the same component:

      \scalebox{\BRANCHEDFACTOR}
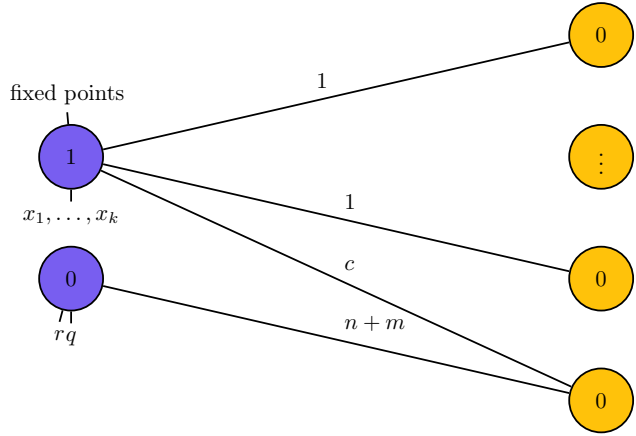
{
      \begin{tikzpicture}[thick,amat/.style={matrix of nodes,nodes in empty cells,
  row sep=2.5em,rounded corners,
  nodes={draw,solid,circle,minimum size=1.0cm}},
  dmat/.style={matrix of nodes,nodes in empty cells,row sep=2.5em,nodes={minimum size=1.0cm},draw=myred},
  fsnode/.style={fill=myblue},
  ssnode/.style={fill=mygreen}]

  \matrix[amat,nodes=fsnode] (mat1) {$1$\\
    $0$\\};

 \matrix[amat,right=7cm of mat1,nodes=ssnode] (mat2) {$0$\\
   \vdots\\
   $0$ \\
 $0$\\};

 \draw  (mat1-1-1) edge["$1$"] (mat2-1-1)
 (mat1-1-1) edge["$1$"] (mat2-3-1)
 (mat1-1-1) edge["$c$"] (mat2-4-1)
 (mat1-2-1) edge["$n+m$"] (mat2-4-1);

 \draw (mat1-1-1) -- +(95:.7) node[anchor=south] {fixed points}
 (mat1-1-1) -- +(270:.7) node[anchor=north] {$x_1,\dots,x_k$}
 (mat1-2-1) -- +(270:.7) node[anchor=north] {$q$}
 (mat1-2-1) -- +(255:.7) node[anchor=north] {$r$};

      \end{tikzpicture}}

      There must be at least one additional branch point from the lower left component and at least two additional branch
      points from the lower right component, a contradiction. Therefore, $q$ and $r$ lie on separate components:

        \scalebox{\BRANCHEDFACTOR}
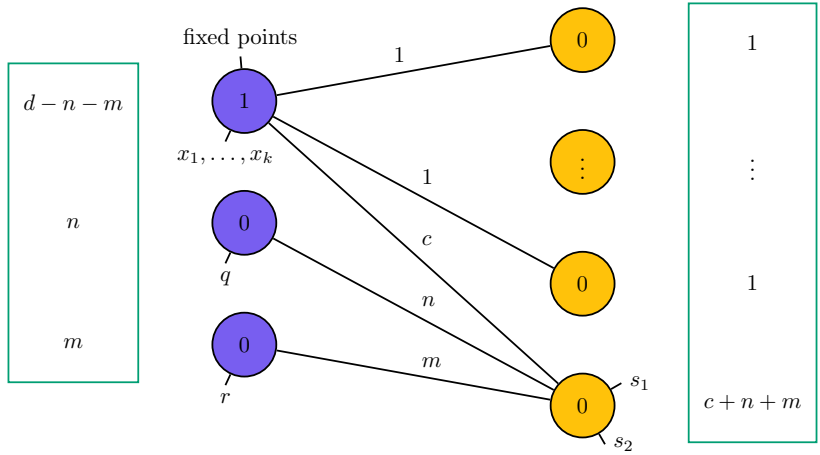
{
          \begin{tikzpicture}[thick,amat/.style={matrix of nodes,nodes in empty cells,
  row sep=2.5em,rounded corners,
  nodes={draw,solid,circle,minimum size=1.0cm}},
  dmat/.style={matrix of nodes,nodes in empty cells,row sep=2.5em,nodes={minimum size=1.0cm},draw=myred},
  fsnode/.style={fill=myblue},
  ssnode/.style={fill=mygreen}]

  \matrix[amat,nodes=fsnode] (mat1) {$1$\\
    $0$\\
    $0$\\};

    \matrix[dmat,left=1cm of mat1] (degrees1) {$d-n-m$\\
    $n$\\
  $m$\\};

 \matrix[amat,right=4cm of mat1,nodes=ssnode] (mat2) {$0$\\
   \vdots\\
   $0$ \\
   $0$\\};

     \matrix[dmat,right=1cm of mat2] (degrees2) {$1$\\
    \vdots \\
    $1$\\
     $c+n+m$\\};

 \draw  (mat1-1-1) edge["$1$"] (mat2-1-1)
 (mat1-1-1) edge["$1$"] (mat2-3-1)
 (mat1-1-1) edge["$c$"] (mat2-4-1)
 (mat1-2-1) edge["$n$"] (mat2-4-1)
 (mat1-3-1) edge["$m$"] (mat2-4-1);

 \draw (mat1-1-1) -- +(95:.7) node[anchor=south] {fixed points}
 (mat1-1-1) -- +(245:.7) node[anchor=north] {$x_1,\dots,x_k$}
 (mat1-2-1) -- +(245:.7) node[anchor=north] {$q$}
 (mat1-3-1) -- +(245:.7) node[anchor=north] {$r$};

 \draw  (mat2-4-1) -- +(30:.7) node[anchor=west] {$s_1$}
 (mat2-4-1) -- +(300:.7) node[anchor=west] {$s_2$};

          \end{tikzpicture}}
          
          By Riemann-Hurwitz for the lower-right component,
          \[
          (c+n+m-3)+(z_1-1)+(z_2-1)=2(c+n+m)-2\implies c=z_1+z_2-n-m-3
          \]
          yielding the second term of the claim. The additional factors arise because
          the edges from the bottom right component should be distinguished, while $s_1$ and
          $s_2$ should not be distinguished.
\end{proof}

It remains to calculate $S_{\{a_1,a_2,a_3\},\emptyset}(d)$ as a base case:

\begin{thm}
  \label{thm:Sabbase}
  When $\a=\{a_1,a_2,a_3\}$, $S_{\a,\emptyset}(d)$ is given by
  \[
  \sum_{i=1}^3\sum_{e+f=d-a_i+2}\frac{2(a_j+a_k-2) H((d),(a_i,1,\dots),(e,f,1,\dots)) H((e,f),(a_j,1,\dots),(a_k,1,\dots)) |\Aut(e,f,1^{d-e-f})|}{(a_i-2)!|\Aut(a_1,a_2,a_3)|}
  \]
  where $\{1,2,3\}=\{i,j,k\}$.
\end{thm}
\begin{proof}
  We will want to calculate $W_{(a_1,1,\dots),(a_2,1,\dots),(a_3,1,\dots)}(d)$ and then divide by \[(d-a_1)!(d-a_2)!(d-a_3)!|\Aut(a_1,a_2,a_3)|\]
  Let $A$ be the standard graph for genus reduction; we will apply Theorem~\ref{thm:admissible}. Any admissible cover of graphs $\Gamma\to\Gamma'$ will
  have a single genus $0$ component on the left side containing
  $x_1$ with a double edge to a genus $0$ component on the right side.
  By Riemann-Hurwitz, the left component has another ramified
  point $p_{i,1}$, and the genus $0$ component to which it has a double
  edge has two ramified points $p_{j,1}$ and $p_{k,1}$. Our picture looks like:

    \scalebox{\BRANCHEDFACTOR}
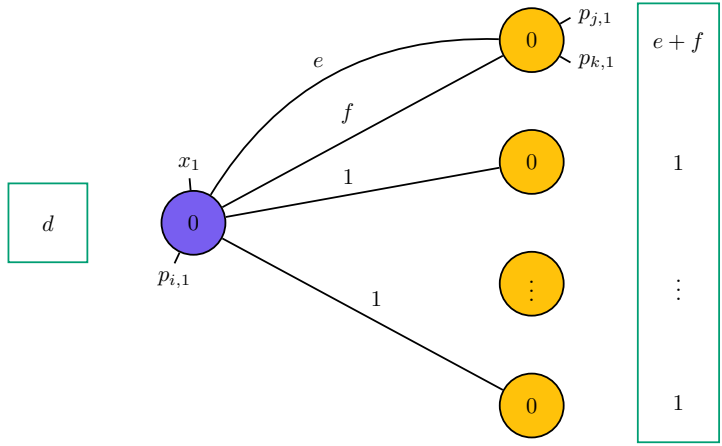
{
            \begin{tikzpicture}[thick,amat/.style={matrix of nodes,nodes in empty cells,
  row sep=2.5em,rounded corners,
  nodes={draw,solid,circle,minimum size=1.0cm}},
  dmat/.style={matrix of nodes,nodes in empty cells,row sep=2.5em,nodes={minimum size=1.0cm},draw=myred},
  fsnode/.style={fill=myblue},
  ssnode/.style={fill=mygreen}]

  \matrix[amat,nodes=fsnode] (mat1) {$0$\\};

    \matrix[dmat,left=1cm of mat1] (degrees1) {$d$\\};

 \matrix[amat,right=4cm of mat1,nodes=ssnode] (mat2) {$0$\\
      $0$\\
      $\vdots$\\
  $0$\\};

 \matrix[dmat,right=1cm of mat2] (degrees2) {$e+f$\\
   $1$ \\
    \vdots \\
    $1$\\};

 \draw  (mat1-1-1) edge["$e$",bend left] (mat2-1-1)
 (mat1-1-1) edge["$f$"] (mat2-1-1)
 (mat1-1-1) edge["$1$"] (mat2-2-1)
 (mat1-1-1) edge["$1$"] (mat2-4-1);

 \draw (mat1-1-1) -- +(95:.7) node[anchor=south] {$x_1$}
 (mat1-1-1) -- +(245:.7) node[anchor=north] {$p_{i,1}$};

 \draw  (mat2-1-1) -- +(30:.7) node[anchor=west] {$p_{j,1}$}
 (mat2-1-1) -- +(330:.7) node[anchor=west] {$p_{k,1}$};

            \end{tikzpicture}}

            Riemann-Hurwitz says that
            \[
            (a_i-1)+(d-1)+(e+f-2)=2d-2\iff e+f=d-a_i+2
            \]
            and
            \[
            (e+f-2)+(a_j-1)+(a_k-1)=2(e+f)-2\iff a_i+a_j+a_k=d+4
            \]
            where this second equality is already known.

            The number of choices for points not marked in this
            picture is computed as follows: first,
            there are $\binom{d-a_j}{e+f-a_j}$ choices
            for which of $p_{j,2},\dots$ to place on the top right component,
            and similarly $\binom{d-a_k}{e+f-a_k}$ choices
            for which of $p_{k,2},\dots$ to place on the top right component.
            Once these are fixed, each other genus $0$ component
            on the right can be identified by which element of
            $p_{j,2},\dots$ it contains, and there are
            then $(d-e-f)!$ choices for placement of
            the remaining $p_{k,2},\dots$ among
            these components.

            Therefore, the contribution to the right side of Theorem~\ref{thm:admissible} is
            \begin{align*}
              &\frac{c_{\Gamma}}{|\Aut(\Gamma)|}\binom{d-a_j}{e+f-a_j}\binom{d-a_k}{e+f-a_k}(d-e-f)!\cdot \\
              &\overline H((d),(a_i,1,\dots),(e,f,1,\dots))\cdot\overline H((e,f),(a_j,1,\dots),(a_k,1,\dots)) \\
              &=\frac{e+f}{|\Aut(e,f)|}\cdot\frac{(d-a_j)!(d-a_k)!}{(e+f-a_j)!(e+f-a_k)!((d-e-f)!)^2}\cdot (d-e-f)!\cdot H((d),(a_i,1,\dots),(e,f,1,\dots))\cdot \\
              &(d-a_i)!|\Aut(e,f,1,\dots)|H((e,f),(a_j,1,\dots),(a_k,1,\dots))\cdot (e+f-a_j)!(e+f-a_k)!|\Aut(e,f)| \\
              &=\frac{(e+f)\prod\limits_{\ell=1}^3(d-a_{\ell})!\cdot H((d),(a_i,1,\dots),(e,f,1,\dots))H((e,f),(a_j,1,\dots),(a_k,1,\dots)) |\Aut(e,f,1,\dots)|}{(d-e-f)!}
            \end{align*}
            and so the contribution to $S_{\a,\emptyset}(d)$ is double this
            (given $|\Aut(A)|=2$) divided by the normalizing factor, i.e.
            \[
            \frac{2(e+f)\cdot H((d),(a_i,1,\dots),(e,f,1,\dots))\cdot H((e,f),(a_j,1,\dots),(a_k,1,\dots))\cdot|\Aut(e,f,1,\dots)|}{(d-e-f)!\cdot |\Aut(a_i,a_j,a_k)|}
            \]
            Since $e+f=a_j+a_k-2$ and $d-e-f=a_i-2$, this proves the claim.
\end{proof}

\subsection{Applying recursions}

We now prove Theorem~\ref{thm:Sabcomputable}:

\begin{proof}
  Let $X=\{(a,b,k)\}=\Z_{\geq 0}^3$ with the total order where
  $(a,b,k)<(a',b',k')$
  if:
  \begin{itemize}
  \item $k<k'$, or
  \item $k=k'$ and $b<b'$, or
    \item $k=k'$, $b=b'$, and $a<a'$
  \end{itemize}
  We will prove by induction on $(a,b,k)$, under this order, that
  any $S_{\a,\b}(\mu)$ with $\ell(\a)=a$, $\ell(\b)=b$, and $\ell(\mu)=k$
  can be computed by a finite sequence of applications
  of Theorems~\ref{thm:reduceb}, \ref{thm:reducea} and \ref{thm:Sabbase}.
  Induction is possible here because for invariants $S_{\a,\b}(\mu)$
  we require $\ell(\mu)\leq d$ and $\ell(\a),\ell(\b)\leq 2d$, implying
  that each $(a,b,k)$ is greater than only finitely many triples
  of interest.
  
  As a base case, if $(a,b,k)<(\infty,0,1)$, then
  this means either $k=0$ (in which case
  our convention is $S_{\a,\b}(\emptyset)=0$) or $k=1$ and $b=0$
  (in which case we can directly apply Theorem~\ref{thm:Sabbase} to find
  $S_{\a,\emptyset}(d)$).

  Now suppose that $(a,b,k)>(\infty,0,1)$
  and that the theorem is true for any $(a',b',k')<(a,b,k)$. There are two cases:
  \begin{enumerate}
  \item If $k>1$, then Theorem~\ref{thm:reducea} reduces $S_{\a,\b}(\mu)$ to
    a sum of terms involving triples less than $(a,b,k)$.
  \item If $k=1$ and $b>0$, then Theorem~\ref{thm:reduceb} reduces
    $S_{\a,\b}(\mu)$ to a sum of terms involving triples less than $(a,b,k)$.
  \end{enumerate}
  In either case the result is proven.
\end{proof}

Define \[
U_{d,a}=S_{\{3\}^{d-2}\cup\{2\}^{5-a},\emptyset}(a,1^{d-a}).
\]
Given a generic $d$-pointed genus $1$ curve $(E,x_1,\dots,x_d)$, $U_{d,1}$
is the number of degree $d$ covers $(E,x_1,\dots,x_d)\to(\P^1,0)$
with $d-2$ unspecified points of $E$ having ramification index $3$.
We now turn to the proof of Theorem~\ref{thm:deg4}:

\begin{proof}
  Firstly, Theorem~\ref{thm:reducea} tells us that
  \[
  U_{d,a}=U_{d,a+1}+U_{d-1,a}\text{ for }a>1,\qquad U_{d,1}=U_{d,2}+U_{d-2,1}
  \]
  This means that the function $f_a(d)=U_{d,a}$ has $f_{a+1}$ as its first differences
  for $a>1$, and $f_1$ has $f_2$ as its first differences with step size $2$.

  We can use Theorem~\ref{thm:Sabbase} to compute:
  \begin{itemize}
  \item $U_{5,5}=f_5(5)=16$
  \item $U_{4,4}=f_4(4)=20$
  \item $U_{3,3}=f_3(3)=8$
    \item $U_{2,2}=f_2(2)=1$
    \end{itemize}
  By Riemann-Hurwitz, $U_{d,a}=f_a(d)=0$ if $a\geq 6$. Each $f_a$ for $a>1$ is determined by $f_{a+1}$ and a value at a point,
  and $f_1$ is determined by $f_2$ and its value at two points (e.g., $f_1(1)=0$ and $f_1(2)=1$).
  Since the quartic in the theorem statement fulfills the same conditions (as can be verified by computer), we are done.

\end{proof}

\appendix

\section{Hurwitz counts}
\label{appendix:hurwitz}

Hurwitz numbers can be interpreted as solutions to counting problems in the symmetric group,
and basic results about these numbers arise from some combinatorial arguments.
We state some results that are used in this paper; proofs are in Appendix A of \cite{Thesis}.

\begin{lem}
  \label{lem:twoscomplete}
For $d=2k$ even,
\[
H((d),(2^k),(2^k),(2,1^{d-2}))=\frac k2
\]
\end{lem}

\begin{lem}
  \label{lem:twoscycles}
  For $d=2k$ even,
  \[
  H((2^k),(2^k),(a,b))=\begin{cases}
  \frac 1d & a=b=k \\
  0 & \text{else}
  \end{cases}
  \]
\end{lem}

\begin{lem}
  \label{lem:twoseven}
For $d=2k$ even,
\[
H((d),(2^k),(2^{k-1},1,1))=\frac 12
\]
\end{lem}

\begin{lem}
  \label{lem:twosodd}
For $d=2k+1$ odd,
\[
H((d),(2^k,1),(2^k,1))=1
\]
\end{lem}

\begin{lem}
  \label{lem:Hurwitz2}
  If $a+b-2=d$, $1\leq n\leq d/2$, and $2\leq a\leq d/2+1$ (equivalently $2\leq a\leq b$), then
\[
H((d-n,n),(a,1,\dots,1),(b,1,\dots,1))=\begin{cases}
\frac 12\min(a-1,n)& d=2n \\
\min(a-1,n)&\text{else}
\end{cases}
\]
\end{lem}

\bibliographystyle{plain}
\bibliography{sources.bib}

\end{document}